\documentclass[10pt]{amsart}
\usepackage[english]{babel}

\usepackage{graphicx} 
\usepackage[utf8]{inputenc}
\usepackage{aliascnt}
\usepackage{hyperref}
\usepackage{enumerate}
\usepackage{physics}
\usepackage{mathrsfs}
\usepackage{multicol}
\usepackage{subcaption}
\usepackage{float}
\usepackage{csquotes}
\usepackage{tikz-cd}
\usepackage{bbm}
\usepackage{tablefootnote}

\usepackage{amsmath}
\usepackage{amsthm}
\usepackage{amssymb}

\usepackage{tikz}
\usetikzlibrary{calc,positioning,shapes}

\usepackage{mathtools}

\theoremstyle{plain}
\newtheorem{theorem}{Theorem}[section]

\newaliascnt{corollary}{theorem}
\newtheorem{corollary}[corollary]{Corollary}
\aliascntresetthe{corollary}

\newaliascnt{corollaryx}{theoremx}

\aliascntresetthe{corollaryx}

\newtheorem{question}{Question}

\addto\extrasenglish{%
}

\newaliascnt{lemma}{theorem}
\newtheorem{lemma}[lemma]{Lemma}
\aliascntresetthe{lemma}

\newaliascnt{lemmax}{theoremx}

\aliascntresetthe{lemmax}

\newtheorem*{lemma*}{Lemma}

\newaliascnt{proposition}{theorem}
\newtheorem{proposition}[proposition]{Proposition}
\aliascntresetthe{proposition}

\newaliascnt{propositionx}{theoremx}

\aliascntresetthe{propositionx}

\newaliascnt{example}{theorem}
\newtheorem{example}[example]{Example}
\aliascntresetthe{example}

\newaliascnt{conjecture}{theorem}
\newtheorem{conjecture}[conjecture]{Conjecture}
\aliascntresetthe{conjecture}

\newaliascnt{observation}{theorem}
\newtheorem{observation}[observation]{Observation}
\aliascntresetthe{observation}

\theoremstyle{definition}

\newaliascnt{definition}{theorem}
\newtheorem{definition}[definition]{Definition}
\aliascntresetthe{definition}

\newtheorem*{definition*}{Definition}

\newaliascnt{construction}{theorem}
\newtheorem{construction}[construction]{Construction}
\aliascntresetthe{construction}

\theoremstyle{remark}
\newaliascnt{remark}{theorem}
\newtheorem{remark}[remark]{Remark}
\aliascntresetthe{remark}

\DeclareMathOperator{\SymOP}{Sym}

\DeclareMathOperator{\linspan}{span}
\DeclareMathOperator{\id}{id}
\DeclareMathOperator{\im}{im}

\DeclareMathOperator{\Kar}{Kar}
\DeclareMathOperator{\Mat}{Mat}
\DeclareMathOperator{\Hom}{Hom}
\DeclareMathOperator{\End}{End}

\DeclareMathOperator{\Pol}{\mathcal{P}}
\DeclareMathOperator{\Kh}{Kh}

\newcommand{\K}{\mathbb{K}}
\newcommand{\F}{\mathbb{F}}
\newcommand{\commR}{\mathbbm{k}}

\newcommand{\N}{\mathbb{N}}
\newcommand{\Z}{\mathbb{Z}}
\newcommand{\Q}{\mathbb{Q}}

\newcommand{\Vect}{\mathrm{Vect}}

\newcommand{\Mod}[1][\commR]{{#1}\mathrm{-Mod}}
\newcommand*{\TL}{\mathrm{TL}}

\newcommand{\Ch}{\mathrm{Ch}}
\newcommand{\cable}{\mathrm{cable}}
\newcommand{\dynkin}{\ensuremath{\leftrightarrow}}
\usepackage{url}
\usepackage{geometry}
\usepackage{tikz}
\usepackage{float}
\usepackage{subcaption}
\usepackage{booktabs}
\author{Karim Ritter von Merkl}

\DeclareMathOperator{\KhR}{KhR}
\DeclareMathOperator{\KhRb}{\mathbf{KhR}}
\DeclareMathOperator{\Cone}{Cone}
\newcommand{\cKh}[1]{\ensuremath{\Kh_{\SymOP^{#1}}}}

\newcommand{\invlabel}{\textit{Inv}}
\newcommand{\coinvlabel}{\textit{Coinv}}
\newcommand{\imlabel}{\textit{Im}}
\newcommand{\coimlabel}{\textit{Coim}}
\newcommand{\homlabel}{\ensuremath{H^\bullet_-}}
\newcommand{\cohomlabel}{\ensuremath{H^\bullet_+}}
\newcommand{\kerlabel}{\textit{Ker}}
\newcommand{\cokerlabel}{\textit{Coker}}

\newcommand{\karjw}[1][n]{\ensuremath{\pi_{#1}}}

\tikzset{anchorbase/.style={baseline={([yshift=-0.5ex]current bounding box.center)}},
	tinynodes/.style={font=\tiny, text height=0.25ex, text depth=0.05ex},
	smallnodes/.style={font=\scriptsize, text height=0.75ex, text depth=0.15ex},
	usual/.style={line width=1.0,color=black, line cap = round},
	thick/.style={line width=1.5,color=black},
	JW/.style={line width=0.5,color=black},
}

\newcommand{\tlcap}{\begin{tikzpicture}[anchorbase, scale=0.15]
		\draw[usual] (-1,0) to[out = 90, in = 180] (0,1) to[out = 0, in =90] (1,0);
\end{tikzpicture}}
\newcommand{\tlcapids}[2]{\begin{tikzpicture}[anchorbase, tinynodes, scale=0.15]
		\draw[usual] (-1.5,0) to (-1.5,1) node[left,xshift=2pt]{$#1$} to (-1.5,2);
		\draw[usual] (-1,0) to[out = 90, in = 180] (0,1) to[out = 0, in =90] (1,0);
		\draw[usual] (1.5,0) to (1.5,1) node[right,xshift=-2pt]{$#2$} to (1.5,2);
\end{tikzpicture}}
\newcommand{\tlcup}{\begin{tikzpicture}[anchorbase, scale=0.15]
		\draw[usual] (-1,0) to[out = 270, in = 180] (0,-1) to[out = 0, in = 270] (1,0);
\end{tikzpicture}}
\newcommand{\tlcupids}[2]{\begin{tikzpicture}[anchorbase, tinynodes, scale=0.15]
		\draw[usual] (-1.5,-2) to (-1.5,-1) node[left,xshift=2pt]{$#1$} to (-1.5,0);
		\draw[usual] (-1,0) to[out = 270, in = 180] (0,-1) to[out = 0, in = 270] (1,0);
		\draw[usual] (1.5,-2) to (1.5,-1) node[right,xshift=-2pt]{$#2$} to (1.5,0);
\end{tikzpicture}}
\newcommand{\tlcupcapids}[2]{\begin{tikzpicture}[anchorbase, tinynodes, scale=0.15]
		\draw[usual] (-1.5,-2.5) to (-1.5,-1) node[left,xshift=2pt]{$#1$} to (-1.5,0.5);
		\draw[usual] (-1,0.5) to[out = 270, in = 180] (0,-0.5) to[out = 0, in = 270] (1,0.5);
		\draw[usual] (-1,-2.5) to[out = 90, in = 180] (0,-1.5) to[out = 0, in =90] (1,-2.5);
		\draw[usual] (1.5,-2.5) to (1.5,-1) node[right,xshift=-2pt]{$#2$} to (1.5,0.5);
\end{tikzpicture}}
\newcommand{\tldensecupcapids}[2]{\begin{tikzpicture}[anchorbase, tinynodes, scale=0.15]
		\draw[usual] (-1.5,-1) to (-1.5,-0.3) node[left,xshift=2pt]{$#1$} to (-1.5,1);
		\draw[usual] (-1,-1) to[out = 90, in = 180] (0,-0.25) to [out = 0, in = 90]  (1,-1);
		\draw[usual]  (-1,1)  to[out = 270, in = 180] (0,0.25) to [out = 0, in = 270] (1,1);
	\draw[usual] (1.5,-1) to (1.5,-0.3) node[right,xshift=-2pt]{$#2$} to (1.5,1);
\end{tikzpicture}}

\newcommand{\tlcrossingids}[2]{\begin{tikzpicture}[anchorbase,tinynodes, scale=0.15]
		\draw[usual] (-1.5,-1) to (-1.5,-0.3) node[left,xshift=2pt]{$#1$} to (-1.5,1);
		\draw[usual] (-1,-1) to [out=90,in=270] (1,1);
		\draw[usual]  (1,-1)  to  [out=90,in=270] (-1,1);
		\draw[usual] (1.5,-1) to (1.5,-0.3) node[right,xshift=-2pt]{$#2$} to (1.5,1);
\end{tikzpicture}}

\newcommand{\positivecrossing}{\begin{tikzpicture}[anchorbase, scale=0.25]
		\draw[->] (1,-1) -- (-1,1);
		\draw[white, ultra thick, line width=3pt] (-1,-1) -- (1,1);
		\draw[->] (-1,-1) -- (1,1);
\end{tikzpicture}}

\newcommand{\negativecrossing}{\begin{tikzpicture}[anchorbase, scale=0.25]
		\draw[->] (-1,-1) -- (1,1);
		\draw[white, ultra thick, line width=3pt] (1,-1) -- (-1,1);
		\draw[->] (1,-1) -- (-1,1);
\end{tikzpicture}}
\newcommand{\unorientedcrossing}{\begin{tikzpicture}[anchorbase, scale=0.25]
	\draw (1,-1) -- (-1,1);
	\draw[white, ultra thick, line width=3pt] (-1,-1) -- (1,1);
	\draw(-1,-1) -- (1,1);
\end{tikzpicture}}

\newcommand{\verticalresolution}{\begin{tikzpicture}[anchorbase, scale=0.25]
		\draw (-1,-1) to[out=45, in=270] (-0.5,0) to[out=90,in =315] (-1,1);
		\draw (1,-1) to[out=135, in=270] (0.5,0) to[out=90,in =225]  (1,1);
\end{tikzpicture}}
\newcommand{\horizontalresolution}{\begin{tikzpicture}[anchorbase, scale=0.25]
		\draw (-1,-1) to[out=45, in=180] (0,-0.5) to[out=0,in =135] (1,-1);
		\draw (-1,1) to[out=315, in=180] (0,0.5) to[out=0,in =225] (1,1);
\end{tikzpicture}}

\begin{document}
	
\title{Computing colored Khovanov homology}
\begin{abstract}
We compare eight versions of finite-dimensional categorifications of the colored Jones polynomial and show that they yield isomorphic results over a field of characteristic zero. As an application, we verify a physics-motivated conjectural formula for colored superpolynomials based on Poincar\'e polynomials of the Khovanov homology of cables. We also obtain a conjectural closed formula for the Poincaré series of the skein lasagna module of $\overline{\mathbb{CP}^2}$. Accompanying this note is an online database of colored superpolynomials.  
\end{abstract}

\maketitle

\section{Introduction}

Khovanov homology is an invariant of links categorifying the Jones polynomial. To each oriented link it associates a bigraded abelian group, computed as the homology of a graded chain complex, such that the graded Euler characteristic recovers the Jones polynomial. Since its introduction by Khovanov \cite{khovanov2000categorification} at the turn of the century, this homology theory of links has been studied extensively. It was quickly discovered to be a stronger invariant than the Jones polynomial \cite{Bar_Natan_2002}: While it is not yet known if the Jones polynomial detects the unknot, Khovanov homology does~\cite{kronheimer2011khovanov}.

The Jones polynomial is naturally part of a whole family of link polynomials, the Reshetikhin--Turaev invariants \cite{Reshetikhin1991} associated to the quantum group $U_q(\mathfrak{sl}_2)$: Let $L$ be a framed, oriented link with $m$ ordered components denoted $K_1,\dots, K_m$. Given the data of a coloring of each component $K_i$ by a natural number $n_i \in \N_0$ (or, more accurately, the $(n_i+1)$-dimensional irreducible representation of $U_q(\mathfrak{sl}_2)$) and setting $\mathbf{n} = (n_1, \dots, n_m)$, one obtains the $\mathbf{n}$-colored Jones polynomial $J_\mathbf{n}(L)\in \Z[q,q^{-1}]$. The usual Jones polynomial arises from the choice $\mathbf{n}=(1,\dots, 1)$. For a framed knot $K$, we can express the $n$-colored Jones polynomial of $K$ in terms of the Jones polynomials of the parallel cables\footnote{The sense of parallelism is determined by the framing. For a blackboard framed diagram, the cable is represented by the $i$-fold parallel push-off.} $K^i$ of $K$ via
\begin{equation}
	\label{eq:coloredJonescableformula}
		J_n \left(K\right) = \sum_{k = 0}^{\lfloor \frac{n}{2} \rfloor} (-1)^k \binom{n-k}{k} J_1 \left( K^{n-2k} \right). 
\end{equation}
This type of relation generalizes in a straightforward multilinear way to colored links.

The colored Jones polynomials appear in several important constructions such as
the WRT-invariants for $3$-manifolds \cite{Reshetikhin1991} and the family as a
whole is the subject of deep and long-standing conjectures such as the AJ
conjecture~\cite{garoufalidis2004characteristicdeformationvarietiesknot}, the
slope conjecture~\cite{Garoufalidis2011-vq}, the volume
conjecture \cite{kashaev1997hyperbolic,
murakami2001colored}, and, more recently, the knots-quivers correspondence \cite{Kucharski_2017, Kucharski_2019} when specializing the colored HOMFLY-PT polynomials. Very roughly, all of these conjectures concern the growth-behavior of colored Jones polynomials at large color.

Shortly after the categorification of the Jones polynomial by Khovanov homology, categorifications of the colored Jones polynomials arose. In his paper~\cite{khovanov2005categorifications}, Khovanov already proposed five different (but---over a field of characteristic zero---conjecturally equivalent) ways of categorifying them. Note that besides the many variations of Khovanov's categorification of the colored Jones polynomial, there are also at least two categorification of the colored Jones polynomial via (bounded-above or bounded-below) categorified idempotents~\cite{cooper2012categorification,Frenkel2012-cb,Rozansky2014-ch}. While the former are always finite-dimensional, the latter are typically infinite-dimensional. In the following, we will focus on Khovanov's finite-dimensional invariant and the relationship between different constructions over a field of characteristic zero.

Categorifying colored link polynomials opens the study of maps between colored homologies that can be functorially associated to link cobordisms embedded in 4-dimensional space. These maps can carry additional information and naturally appear in topological constructions, such as the 4-manifold invariants based on Khovanov homology \cite{Morrison_2022} now known as skein lasagna modules, see \cite{manolescu2022skein,MR4589588,Hogancamp2025-ol}. 
Recently, Ren--Willis~\cite{ren2024khovanovhomologyexotic4manifolds} have shown that the skein lasagna modules detect exotic pairs of 4-manifolds based on purely algebraic data. 

Our work here is motivated by the desire to better understand skein lasagna modules and, more generally, the growth-behaviour of colored Khovanov homology in the color parameter. To this end, we have assembled a database \cite{coloredKh} of Poincaré polynomials of colored Khovanov homology, which is accessible at \url{https://colored-kh.math.uni-hamburg.de}. In the following we explain the method used to compute these superpolynomials.

\subsection{Colored superpolynomials from superpolynomials of cables}
Let $\F$ be a field. For any given finite-dimensional bigraded $\F$-vector space $V$, we denote by $\Pol(V)$ the 2-variable Poincaré polynomial or \emph{superpolynomial} of $V$:
	\[ \Pol(V)(t,q) := \sum_{i,j} \dim_\F\left(V^{i,j} \right) t^i q^j \in \N_0[t, t^{-1}, q, q^{-1}]. \]

If $\F$ is of characteristic zero, then there is an unambiguous notion of
finite-dimensional colored Khovanov homology $\cKh{n}(K)$ for framed oriented
knots $K$ in $B^3$, taking values in isomorphism classes of bigraded $\F$-vector
spaces or, equivalently, their Poincaré polynomials, see Proposition~\ref{prop:main}. 

\begin{theorem}
	\label{thm:introcableformula} Let $K$ be a framed oriented knot in $B^3$ and
$n\in \N_0$. Then the Poincaré polynomial of the finite-dimensional colored
Khovanov homology $\cKh{n}(K)$ can be computed from the Poincaré polynomials of
(ordinary) Khovanov homologies $\Kh(K^i)$ of the parallel cables $K^i$ of $K$
with coefficients in $\F$ via:
	\begin{equation}
		\notag
		\label{eq:intropoincare}
	\Pol \left(\cKh{n}(K)\right) = \sum_{k = 0}^{\lfloor \frac{n}{2} \rfloor} (-1)^k \binom{n-k}{k} \Pol \left( \Kh \left( K^{n-2k} \right) \right) \in \N_0[t, t^{-1}, q, q^{-1}] 
\end{equation}

\end{theorem}

For simplicity, we only state the theorem restricted to knots here. The general version in \autoref{prop:definitionsEquivalent} covers the case of colored links. Interestingly, the relation \eqref{eq:coloredJonescableformula} between the colored
Jones polynomials of a knot and the Jones polynomials of its cables carries over verbatim
to Poincaré polynomials of colored Khovanov homology, which is also true for links if one is careful with orientations. This relationship, in a conjectural form, was already used in the physics literature, see e.g. \cite{gukov2020bpsspectra3manifoldinvariants}.

\begin{corollary}
	\label{cor:inverse}
	In the context of \autoref{thm:introcableformula}, we have the inverse relation
	\begin{equation*}
		\Pol \left( \Kh \left( K^{n} \right) \right) = \sum_{k = 0}^{\lfloor \frac{n}{2} \rfloor} T_{n,n-2k} \Pol \left(\cKh{n-2k}(K)\right)
	\end{equation*}
	where $T_{n,m}$ are the numbers of the Catalan triangle \cite[A053121]{oeis} which count truncated Dyck paths.
\end{corollary}

\begin{remark}
		Unlike the colored Jones polynomials and the projector colored homology \cite{cooper2012categorification, Rozansky2014-ch} that depend on the framing only up to shift, $\cKh{n}(K)$ depends on the framing in a more subtle way. Some examples of this are discussed in \autoref{sec:examples} and even more can be seen in the database \cite{coloredKh}.
\end{remark}

\begin{remark}
	Hogancamp \cite{Hogancamp2018-xq} defined a version of projector colored homology that replaces the unbounded complexes introduced by Cooper--Krushkal \cite{cooper2012categorification} with bounded ones. Since these do not categorify the Jones--Wenzl projectors, but multiples of them, the resulting link invariant categorfies a multiple of the colored Jones polynomial. It would be interesting to understand the relation between these variations.
\end{remark}

\subsection{Versions of colored Khovanov homology}

During the proof of \autoref{thm:introcableformula}, we will naturally relate some of the various notions of colored Khovanov homology mentioned above to each other. We therefore use the opportunity to fill in some details in the proof (of the folklore result) that these variations are indeed isomorphic over a field of characteristic zero as outlined in \cite[Proposition 6.5]{Hogancamp2025-ol} for three of the following variations. Detailed construction for all of these versions can be found in \autoref{def:coloredHomology} and the preceding explanations, but to give an overview already, we sketch those constructions. 

\begin{proposition}
	\label{prop:main} Let $K$ be a framed knot, $n\in \N$, and $\F$ a field of
characteristic zero. Consider the Khovanov homologies $\Kh\left(
K^{i}\right)$ of the alternatingly oriented parallel $i$-cables of $K$ with coefficients in $\F$ for
$0\leq i \leq n$. Then the following incarnations of $n$-colored Khovanov
homology for $K$ with coefficients in $\F$ are isomorphic:
\begin{center}
	\begin{tabular}{ r l}
		(\invlabel) & The invariants for the $S_n$-action on $\Kh\left( K^{n}\right)$ induced by braiding cobordisms \cite{Grigsby_2017}.  \\ 
		(\coinvlabel) &  The coinvariants for the same $S_n$-action on $\Kh\left( K^{n}\right)$.  \\  
		(\cohomlabel) & The cohomology of Khovanov's cochain complex from
		\cite{khovanov2005categorifications} built from $\Kh\left( K^{i}\right)$, which is      \\
		&supported in non-negative degrees.\\
		(\homlabel) &   The cohomology of an analogous cochain complex concentrated in non-positive degrees. \\
		(\kerlabel) & The intersection of kernels of annulus cobordism maps $\Kh\left( K^{n}\right)\to \Kh\left( K^{n-2}\right)$.  \\
		(\cokerlabel) & The joint cokernel of all annulus cobordism maps $\Kh\left( K^{n-2}\right)\to \Kh\left( K^{n}\right)$.  \\
		(\imlabel) & The image of the idempotent endomorphism of $\Kh\left( K^{n}\right)$ corresponding to the $n$-th Jones--Wenzl \\
		& projector $p_n$ under the Temperley--Lieb action given by annulus cobordism maps \cite{Grigsby_2017}.  \\
		(\coimlabel) & The quotient of $\Kh\left( K^{n}\right)$ by the image of the complementary idempotent to $p_n$.   
	\end{tabular}
	\end{center}
 \end{proposition}

 \begin{remark}
	With the exceptions of (\imlabel) and (\coimlabel), all of the above can be defined over $\Z$ or, more generally, any commutative ring $\commR$. The versions (\imlabel) and (\coimlabel) of $n$-colored Khovanov homology can be defined over a commutative ring $\commR$ when the $n$-th Jones--Wenzl projector exists in $\TL^\commR(2)$. Therefore, we restrict to a field of characteristic 0 when we work with these versions. Over a field of characteristic zero, isomorphisms (\invlabel) $\cong$ (\coinvlabel) and (\imlabel) $\cong$ (\coimlabel) are to be expected for purely algebraic reasons, as we will record in \autoref{prop:dualities}. In contrast, these are (in general) not isomorphic over a commutative ring $\commR$ or field of positive characteristic.
 \end{remark}
 
Finally, we will extract some observations from the database \cite{coloredKh} that can be related to conjectures in the literature, derive a conjectural closed formula for the Poincaré series of the skein lasagna module of $\overline{\mathbb{CP}^2}$ and used to define a new family of knot invariants.

\subsection*{Acknowledgments}

The author would like to thank Jesse Cohen, Leon Goertz, Eugene Gorsky, Hannes Knötzele, Lukas Lewark, Qiuyu Ren and Claudius Zibrowius for helpful conversations and comments on earlier versions of this note and his advisor Paul Wedrich for support and encouragement along the way.
I would also like to thank an anonymous referee for many helpful comments and suggestions to improve this work.
The author acknowledges support by the Deutsche Forschungsgemeinschaft (DFG, German Research Foundation) under Germany's Excellence Strategy - EXC 2121 ``Quantum Universe'' - 390833306 and is partially funded by the Collaborative Research Center - SFB 1624 ``Higher structures, moduli spaces and integrability'' - 506632645.

\section{Preliminaries}
\label{sec:prelim}

\subsection{The Temperley--Lieb category}

 Let $\commR$ be a commutative ring. The Temperley--Lieb category $\TL(d) := \TL^\commR(d)$ (see e.g.\ \cite{Kauffman-Lins94}) is the $\commR$-linear category with objects $n \in \N_0$ and $\Hom(n,m)$ freely generated over $\commR$ by crossingless matchings on $n+m$ points. When $n+m$ is even, these can be interpreted diagrammatically as embeddings of $\frac{n+m}{2}$ intervals in a rectangle with $n$ (equidistant, centered) endpoints at the bottom and $m$ at the top up to isotopy relative to the boundary. If $m+n$ is odd, then $\Hom(n,m) = 0$. Composition is modeled on the $\commR$-bilinear extension of stacking diagrams on top of each other and removing every circle in a diagram at the cost of introducing a factor $d$. For instance:
 \begin{gather*}
 	\begin{tikzpicture}[anchorbase,scale=0.5,tinynodes]
 		\draw[usual, out = 90, in = 270] (1,-1) to (-1,1);
 		\draw[usual] (-1,-1)  to[out = 90, in = 180] (-0.5,-0.5) to [out = 0, in = 90] (0,-1);
 		\draw[usual]  (0,1) to[out = 270, in = 180] (0.5,0.5) to [out = 0, in = 270] (1,1);
 	\end{tikzpicture}
 	\circ 
 	\begin{tikzpicture}[anchorbase,scale=0.5,tinynodes]
 		\draw[usual,out = 90, in = 270]  (0,-1) to (1,1);
 		\draw[usual] (-1,1) to[out = 270, in = 180] (-0.5,0.5) to [out = 0, in = 270] (0,1);
 	\end{tikzpicture}
 	=
 	\begin{tikzpicture}[anchorbase,scale=0.5,tinynodes]
 		\draw[usual](0,-3) to[out = 90, in = 270] (1,-1) to[out = 90, in = 270]  (-1,1);
 		\draw[usual] (-1,-1) to[out = 90, in = 180] (-0.5,-0.5) to [out = 0, in = 90] (0,-1)  to [out = 270, in = 0]  (-0.5,-1.5) to[out = 180, in = 270] (-1,-1);
 		\draw[usual] (0,1) to[out = 270, in = 180] (0.5,0.5) to [out = 0, in = 270] (1,1);
 	\end{tikzpicture}
 	=
 	d \cdot
 	\begin{tikzpicture}[anchorbase,scale=0.5,tinynodes]
 		\draw[usual, out = 90, in = 270] (0,-1) to (-1,1);
 		\draw[usual]  (0,1) to[out = 270, in = 180] (0.5,0.5) to [out = 0, in = 270] (1,1);
 	\end{tikzpicture}
 \end{gather*}
 Horizontal juxtaposition of diagrams turns the Temperley--Lieb category into a (strict) monoidal category with $0 \in \N_0$ being the monoidal unit. Every object $n \in \TL(d)$ is self-dual with the (co)evaluations being $n$ nested caps/cups. Therefore, the duals of morphisms are given by a 180-degree rotation of diagrams. Reflecting diagrams in a vertical line in the plane yields an (opmonoidal) automorphism of the Temperley--Lieb category denoted by $(-)^\dynkin$.
 
 The $n$-th Temperley--Lieb algebra $\TL_n(d)$ is the endomorphism algebra $\End(n)$ in $\TL(d)$.
 We may look for a family of special idempotents in these algebras whose existence depends on the units of the ring $\commR$.
 \begin{definition}
 	\label{def:JW}
 	An element $p_n \in \TL_n(d)$ satisfying the three conditions
 	\begin{enumerate}[(i)]
 		\item $p_n^2 = p_n$.
 		\item $p_n$ annihilates cups/caps, i.e.\ when denoting $p_n$ by $\begin{tikzpicture}[anchorbase, scale=0.2]
 			\draw[JW,fill=white] (-2.5,-2.5) rectangle (2.5,-4.5) node[pos=0.5] {$n$};
 		\end{tikzpicture}$, we have
 		\begin{gather*}
 			\begin{tikzpicture}[scale=0.25, anchorbase]
 				\draw[usual] (-1.5,-2.5) to (-1.5,-1) node[left,xshift=2pt]{$i$} to (-1.5,0.5);
 				\draw[usual] (-1,-2.5) to[out = 90, in = 180] (0,-1.5) to[out = 0, in =90] (1,-2.5);
 				\draw[usual] (1.5,-2.5) to (1.5,-1) node[right,xshift=-2pt]{} to (1.5,0.5);
 				\draw[JW,fill=white] (-2.5,-2.5) rectangle (2.5,-4.5) node[pos=0.5] {$n$};
 			\end{tikzpicture}
 			= 0 =
 			\begin{tikzpicture}[scale=0.25, anchorbase]
 				\draw[usual] (-1.5,-7.5) to (-1.5,-6) node[left,xshift=2pt]{$i$} to (-1.5,-4.5);
 				\draw[usual] (-1,-4.5) to[out = 270, in = 180] (0,-5.5) to[out = 0, in =270] (1,-4.5);
 				\draw[usual] (1.5,-7.5) to (1.5,-6) node[right,xshift=-2pt]{} to (1.5,-4.5);
 				\draw[JW,fill=white] (-2.5,-2.5) rectangle (2.5,-4.5) node[pos=0.5] {$n$};
 			\end{tikzpicture}.
 		\end{gather*}
 		\item The coefficient of the identity in $p_n$ (with respect to the crossingless matching basis) is one.
 	\end{enumerate} 
 	is called \emph{$n$-th Jones--Wenzl projector}.
 \end{definition}
 If such an element exists, it is in fact uniquely characterized by the three properties above and non-zero (see e.g.\ \cite{Kauffman-Lins94}).
If we work over the ring $\Z[q^{-1},q]]$ and use $d = q + q^{-1}$, we can express them recursively via
\begin{gather*}\label{eq:jw-recursion}
	p_0
	=\emptyset\, ,
	\quad
	p_1 =
	\begin{tikzpicture}[anchorbase,scale=.33]
		\draw[usual] (0,0) to (0,2);
	\end{tikzpicture}\,
	,\quad
	p_n
	=
	\begin{tikzpicture}[anchorbase,scale=.33]
		\draw[JW] (-1.5,0) rectangle (1.5,2);
		\node at (0,.9) {$n$};
	\end{tikzpicture}
	:=
	\begin{tikzpicture}[anchorbase,scale=.33]
		\draw[JW] (-1.5,0) rectangle (1.5,2);
		\node at (0,.9) {$n-1$};
		\draw[usual] (2,0) to (2,2);
	\end{tikzpicture}
	-
	\frac{[n-1]}{[n]}\cdot
	\begin{tikzpicture}[anchorbase,scale=.33]
		\draw[JW] (-2.5,1) rectangle (2.5,3);
		\node at (0,1.9) {$n-1$};
		\draw[JW] (-2.5,-1) rectangle (0.5,1);
		\node at (-1,-.1) {$n-2$};
		\draw[usual] (1.5,1) to[out=270, in=180] (2.25,.25) to[out=0, in=270] (3,1) to (3,3);
		\draw[usual] (1.5,-1) to[out=90, in=180] (2.25,-.25) to[out=0, in=90] (3,-1) to (3,-3);
		\draw[JW] (-2.5,-1) rectangle (2.5,-3);
		\node at (0,-2.1) {$n-1$};
	\end{tikzpicture}
	\quad
	\text{ if }n \geq 2,
\end{gather*}

where $[n] = \frac{q^{n} - q^{-n}}{q - q^{-1}}$ denotes the $n$-th quantum integer. In this case, all the necessary inverses exist such that this recursion is well-defined. The same holds true when we work over a field of characteristic zero and specify $q = 1$.

Usually, we will not work with the Temperley--Lieb category $\TL := \TL(2)$ directly, but rather with the additive completion of its Karoubi envelope $\Mat(\Kar(\TL))$ and with coefficients in a field of characteristic zero. This category consists of finite formal direct sums of objects in the idempotent complete category $\Kar(\TL)$ with matrices of morphisms between them. Composition is modeled on matrix multiplication. 
See e.g.\ \cite{Hogancamp2025-ol} for details on the Karoubi envelope and (a graded versions of) the additive closure.

\section{Variations of colored Khovanov homology}
\label{sec:variation}

\subsection{Conventions}
\label{subsec:convention}

There are many different grading conventions used in Khovanov homology that can make it difficult to compare, relate or combine computations obtained with various methods. Sullivan--Zhang \cite[Section 3.2]{sullivan2024kirbybeltscategorifiedprojectors} give a great overview over various conventions for Khovanov homology and where they are used. We will quickly summarize those below, but highly recommend the original overview. We additionally include (a non-complete) list of programs known to us computing Khovanov homology using the respective conventions. In order to make the database \cite{coloredKh} as agnostic about them as possible, any of these can be chosen as output format of the data and a general converter for Poincaré polynomials using these conventions is provided.

Note that all of the following conventions are cohomological in the sense that differentials raise the $t$-degree. The degree zero part of complexes is underlined.
\subsection*{Khovanov's original formulation}
	Khovanov \cite{khovanov2000categorification} originally introduced the categorification of the Jones polynomial based on the oriented skein relations:
	\[
		\positivecrossing = q \ \underline{\verticalresolution} \to q^2 \ \horizontalresolution
		\quad \text{and} \quad
		\negativecrossing = q^{-2} \ \horizontalresolution \to q^{-1} \ \underline{\verticalresolution}
		\]
	This convention is equivalent to the one used by Bar-Natan \cite{Bar_Natan_2002, Bar_Natan_2005} based on the unoriented skein relation:
	\[ \unorientedcrossing = \underline{\verticalresolution} \to q \ \horizontalresolution  \]
	together with a global shift on the resulting complex that reproduces the grading of the oriented skein relations.
	We denote this version of Khovanov homology by $\Kh$. It is invariant under change of framing or, equivalently, Reidemeister I moves, up to bigrading-preserving isomorphism.
	
	This is the convention used in the programs \texttt{FastKh} (written by Dror Bar-Natan), \texttt{JavaKh} (written by Jeremy Green, see \cite{barnatan2007fastkhovanovhomologycomputations} for both) and  \texttt{KnotJob} (written by Dirk Schütz \cite{Knotjob}). The first two programs are also available through the \texttt{Mathematica} package \texttt{KnotTheory'} \cite{KnotTheory}. Therefore, this convention is also the one displayed in the KnotAtlas \cite{Knotatlas}. Finally, when displaying results in \autoref{sec:observations}, we also stick to this convention.
	\subsection*{Khovanov's formulation for tangles}
	In tangle-related literature, the convention from Khovanov \cite{Khovanov_2002, khovanov2003patterns} with $q$ and $q^{-1}$ interchanged is used most often. This yields the skein relations:
	\[
	\positivecrossing = q^{-1} \ \underline{\verticalresolution} \to q^{-2} \ \horizontalresolution
	\quad \text{and} \quad
	\negativecrossing = q^{2} \ \horizontalresolution \to q^{1} \ \underline{\verticalresolution}
	\]
	and is also insensitive to framing changes/Reidemeister I moves. We
	denote this version by $\overline{\Kh}$ and get the isomorphism
	$\overline{\Kh}^{i,j}(L) \cong \Kh^{i,-j}(L)$.
	\subsection*{(Unframed) Khovanov--Rozansky \texorpdfstring{$\mathfrak{sl}_2$}{sl2} homology} 
	Khovanov homology also arises
	(over $\Q$) as the $N=2$ case of Khovanov's and Rozansky's $\mathfrak{sl}_N$
	homology \cite[p.10]{khovanov2008matrix}, denoted by
	$\KhRb_2$. In this convention, the skein relations read:
	\[
	\positivecrossing = q^{2} \ \horizontalresolution \to q \ \underline{\verticalresolution}
	\quad \text{and} \quad
	\negativecrossing = q^{-1} \ \underline{\verticalresolution} \to q^{-2} \ \horizontalresolution
	\]
	Compared to the previous skein relations, we exchange positive with negative crossings and possibly $q$ with $q^{-1}$. This yields the isomorphisms $\KhRb_2^{i,j}(L) \cong \overline{\Kh}^{i,j}(L^!) \cong \Kh^{i,-j}(L^!)$, where $L^!$ is the mirror image of the link $L$. This convention is again insensitive to framing/Reidemeister I.
	
	This convention is used in the program \texttt{Khoca} (written by Lukas Lewark and Andrew Lobb \cite{Lewark_2016,Khoca}) since this computes $\mathfrak{sl}_N$ homology. \texttt{Khoca} was used to compute the results in our database \cite{coloredKh}, even though we prefer to show results in \autoref{sec:observations} in the convention of Khovanov's original formulation.
	
	\subsection*{Framed Khovanov--Rozansky \texorpdfstring{$\mathfrak{sl}_2$}{sl2} homology}
	 Other than the previous conventions, there is also a version of Khovanov--Rozansky homology introduced by \cite{Morrison_2022} which is sensitive to framing changes (resp.\ Reidemeister I moves or the writhe $w(D)$ of a diagram $D$ for the link $L$). We denote this version by $\KhR_2$. It corresponds to the skein relations:
	\[
	\positivecrossing = q \ \horizontalresolution \to \underline{\verticalresolution}
	\quad \text{and} \quad
	\negativecrossing = \underline{\verticalresolution} \to q^{-1} \ \horizontalresolution
	\]
	This yields $\KhR_2(L) \cong \KhRb_2(L)\{-w(D)\}$ or explicitly 
	$$\KhR_2^{i,j}(L) \cong \KhRb_2^{i,j+w(D)}(L)$$
	 for a link $L$ with diagram $D$ such that the framing of $L$ is the blackboard framing of $D$.

\subsection{Defining variations of colored Khovanov homology}
\label{subsec:definitions}

We start with preparatory work to state the various definitions of colored Khovanov homology sketched in the introduction for knots and extend the constructions to links later. 
Let $\commR$ be a commutative ring and denote by $\Kh(L;\commR)$ the Khovanov homology of a link $L$ with coefficients in $\commR$. Most of the time, the choice of commutative ring $\commR$ will be clear from context, so we suppress it in the notation and use $\Kh(L)$ instead. Fix a framed, oriented knot $K$ and consider its $n$-cable $K^n$ by introducing parallel copies along the framing with alternating orientations of components such that the first component carries the orientation of $K$.

We get an action of the braid group $\mathrm{Br}_{n}$ on the Khovanov homology of the cable $\Kh(K^n)$ by considering component interchanging cobordisms and their induced cobordism maps. Note that we require a fully functorial version of Khovanov homology for this throughout \cite{Caprau2008,Clark_2009,Blanchet2010,Ehrig_2018}. At this point, we need to be careful, since interchanging components will not preserve the chosen orientation of $K^n$. Reversing the orientation of a component \enquote{just} introduces a constant grading shift on Khovanov homology. For cables, this only depends on the total number of components with (anti-) parallel orientation. Thus, we can reorient for free and think of the cobordism map as an honest endomorphism of $\Kh(K^n)$ with correct orientations.

By a result of Grigsby--Licata--Wehrli \cite{Grigsby_2017}, this action of the braid group factors through the symmetric group $S_n$. Initially, this was stated in a setting of annular links and the Temperley--Lieb category with circle value $-2$, but the same argument works in the usual setting. A spelled-out version for sufficiently functorial Khovanov--Rozansky type A link homologies can be found in \cite{Gorsky2023-ih}. Further, we do not need to restrict the commutative ring we are working over for this to hold true. From this $S_n$-action, we will eventually define two versions of colored Khovanov homology.

To construct the other versions, we adapt the point of view of \cite{Hogancamp2025-ol} and think of Khovanov homology of the fixed oriented, framed knot $K$ with coefficients in $\commR$ as a functor
$$ \Kh_K^\commR =:  \Kh_K  \colon \TL^\commR(2) =: \TL \to \Mod^{\Z \times \Z}.$$
The cabling functor $\Kh_K$ sends an object $n \in \TL$ of the Temperley--Lieb category to the bigraded $\commR$-module $\Kh(K^n)$. 
The action on morphisms is generated by the maps corresponding to $\tlcapids{k}{\ell} \in \Hom_\TL(k+\ell+2,k+\ell)$ and  $\tlcupids{k}{\ell} \in \Hom_\TL(k+\ell,k+\ell+2)$---or in words---caps/cups between all possible numbers of identity stands. The former is sent to the map on Khovanov homology induced by the annulus between two adjacent components at the position of the cap and leaving all other components unchanged. Similarly, the latter creates two new copies of $K$ with alternating orientations at the position of the cup. Since the orientation of $K^n$ alternates between components, both of them are well-defined. The cobordism associated to $\Kh_K\left(\tlcapids{k}{\ell}\right)$ can be described as the composition of a saddle that merges the two components followed by a series of isotopies to separate the resulting loop from the link and finally a cap to contract it. The cobordism representing the morphism $\Kh_K\left(\tlcupids{k}{\ell}\right)$ enjoys a similar description. 

From the equation $\tlcap \circ \tlcup = 2$ we obtain that if $2$ is not a
zero-divisor in $\commR$, $\Kh_K\left(\tlcupids{k}{\ell}\right)$ is injective.
The same arguments applies in the presence of additional identity strands. This
uses that every torus, regardless of embedding, acts as multiplication by $2$ in
Khovanov homology \cite{rasmussen2005khovanovsinvariantclosedsurfaces}
\footnote{This is our main motivation to work with $\TL = \TL(2)$. The sign
choices used for the functoriality fix in \cite{Grigsby_2017} sets the circle
parameter to $-2$, see also \cite[Proof of Prop
4.4]{ren2024leefiltrationstructuretorus} for a discussion. The results of our work are agnostic about the sign of the circle parameter, so in this paper, we use $+2$ for
concreteness.}. Since $\Mod^{\Z \times \Z}$ is idempotent complete and additive,
we extend this functor to the additive closure of the Karoubi envelope
$\Mat(\Kar(\TL))$.
If $e \in \End_\TL(n)$ is an idempotent, $\Kh_K(e) \in \End(\Kh(K^n))$ is a projection onto the submodule $\im \Kh_K(e)$. The extension of $\Kh_K$ to the Karoubi envelope maps the object corresponding to the idempotent $e$ to the module $\im \Kh_K(e)$.
Note that since the Khovanov homology of every link is
finitely generated, the image of $\Kh_K$ consists only of finitely generated
$\commR$-modules. Similarly, we can extend $\Kh_K$ to chain complexes over $\Mat(\Kar(\TL))$. We identify all these versions of $\Kh_K$ along the inclusions $\TL \hookrightarrow \Mat(\Kar(\TL)) \hookrightarrow \Ch(\Mat(\Kar(\TL)))$, where the last arrow considers an object as a chain complex concentrated in homological degree zero. Since $\Kh_K$ is additive, it preserves chain homotopy equivalence.

For the remaining versions of $n$-colored Khovanov homology, we will need the cochain complex $C_n$. A convenient way to construct it is to first describe an auxiliary diagram category $\Gamma_n$ and a cochain complex $\overline{C}_n$ over it:

\begin{construction}
	\label{cons:cube} 
	Let $n \in \N$. A composition of $n$ is a tuple of positive integers summing to $n$.
	Recall that the set of compositions of $n$ admits a partial order by coarsening. Concretely, let $s = (s_1, s_2, \dots, s_k)$ be a composition of $n$, i.e.\ $n = \sum_{i =1}^k s_k$, then for all $1 \leq i \leq k-1 $ we have another composition $s'_i=(s_1, \dots, s_i + s_{i+1}, s_{i+2}, \dots, s_k )$ given by adding the $i$-th and $i+1$-th term of $s$. The partial order on the set of compositions is then reflexively and transitively generated by the relations $s \leq s'_i$.
		
	Denote by $\Gamma_n$ the subposet of compositions that only use $1$ and $2$. In the $\Z$-linearization of this category, the abelian group of morphisms $\Hom(s,s')$ from $s$ to $s'$ is either isomorphic to $\{0\}$ if $s \not \leq s'$ or to $\Z$ if $s \leq s'$. We fix isomorphisms such that composition is given by multiplication in $\Z$. Since every square in a poset category commutes, a square in the $\Z$-linearization commutes if and only if the multiplication of integers gives the same result. Finally, we form the additive closure to obtain the category over which we now construct the cochain complex $\overline{C}_n$ with differential $\overline{d}$ as follows: In homological degree $k$ we place the direct sum of all compositions in $\Gamma_n$ containing 2 exactly $k$ times. The $k$-th component of $\overline{d}$ is given in terms of maps between components $s$ in degree $k$ and $s'$ in degree $k+1$. Note that the number of twos increases by exactly one when following the differential:
	\begin{itemize}
		\item If $s \not \leq s'$ as compositions, the differential has to be $0$ between these components.
		\item If  $s \leq s'$ as compositions, then there is one unique new two in $s'$ relative to $s$. Denote the number of twos to the left of this one by $\ell$ and set the morphism between these components to $(-1)^\ell$.
		Without these signs, all the squares without zero morphisms would commute as mentioned above. Including the signs makes all such squares contain an odd number of signs, i.e.\ anti-commute. Therefore, $\overline{d} \circ \overline{d} = 0$ is indeed satisfied. Any sign placement making every square without zero morphisms anti-commute would yield an isomorphic complex as discussed in \cite[Lemma 2.1]{Beliakova_Wehrli_2008}.
	\end{itemize}
	
\end{construction}

\begin{remark}
	\label{rem:subcomplex}
	The category $\Gamma_n$ is called the \emph{Fibonacci cube}. By realizing it as a subset of the $(n-1)$-dimensional hypercube of all compositions, $\Gamma_n$ naturally inherits its poset structure \cite{205649}. Using this point of view, we can alternatively describe $\overline{C}_n$ as the totalization of the $\Z$-linearization of $\Gamma_n$ considered as a diagram which is isomorphic by \cite[Lemma 2.1]{Beliakova_Wehrli_2008}. We could have also taken the totalization of the whole $\Z$-linearized cube of compositions and consider the quotient by the subcomplex containing numbers larger than $2$ in compositions to construct an isomorphic cochain complex.
\end{remark}

\begin{example}
	\label{ex:cbar}
	In \cite{khovanov2005categorifications, Beliakova_Wehrli_2008}, the complex $\overline{C}_n$ is described in terms of \emph{paired dots}. In order to recover this way of constructing $\overline{C}_n$, we visualize objects of $\Gamma_n$ as configurations of dots and barbells (or unpaired and paired dots), where 1 corresponds to $\begin{tikzpicture}[anchorbase]
		\filldraw[black] (0,0) circle (2pt);
	\end{tikzpicture}$ and 2 corresponds to $\begin{tikzpicture}[anchorbase]
		\filldraw[black] (0,0) circle (2pt);
		\filldraw[black] (0.5,0) circle (2pt);
		\draw[ultra thick] (0,0) -- (0.5,0);
	\end{tikzpicture}$.
	For $n =4$ the complex $\overline{C}_n$ from \autoref{cons:cube} takes the form on the left-hand side of \autoref{fig:gammacomplex}. The right-hand side shows the full cube of compositions using paired configurations of three and four dots analogously to represent the numbers $3$ and $4$. Totalizing and forming the quotient by the gray subcomplex (indicated by the gray dot in the subscript) yields an isomorphic complex as mentioned in \autoref{rem:subcomplex}.
	\begin{figure}[H]
		\centering
		\begin{gather*}
		\begin{tikzpicture}[anchorbase]
			\filldraw[black] (0,0) circle (2pt);
			\filldraw[black] (0.5,0) circle (2pt);
			\filldraw[black] (1,0) circle (2pt);
			\filldraw[black] (1.5,0) circle (2pt);
			\draw[ultra thick] (1,0) -- (1.5,0);
			\filldraw[black] (0,2) circle (2pt);
			\filldraw[black] (0.5,2) circle (2pt);
			\filldraw[black] (1,2) circle (2pt);
			\filldraw[black] (1.5,2) circle (2pt);
			\draw[ultra thick] (0.5,2) -- (1,2);
			\draw  (0.75,3) node {$\bigoplus$};
			\draw  (0.75,1) node {$\bigoplus$};
			\filldraw[black] (-3,2) circle (2pt);
			\filldraw[black] (-2.5,2) circle (2pt);
			\filldraw[black] (-2,2) circle (2pt);
			\filldraw[black] (-1.5,2) circle (2pt);
			\filldraw[black] (3,2) circle (2pt);
			\filldraw[black] (3.5,2) circle (2pt);
			\filldraw[black] (4,2) circle (2pt);
			\filldraw[black] (4.5,2) circle (2pt);
			\draw[ultra thick] (3,2) -- (3.5,2);
			\draw[ultra thick] (4,2) -- (4.5,2);
			\filldraw[black] (0,4) circle (2pt);
			\filldraw[black] (0.5,4) circle (2pt);
			\filldraw[black] (1,4) circle (2pt);
			\filldraw[black] (1.5,4) circle (2pt);	
			\draw[ultra thick] (0,4) -- (0.5,4);
			\draw[->] (-1.3,2) -- (-0.2,2);
			\draw[->] (-1.5,2.2) -- (0.5, 3.8);
			\draw[->] (-1.5,1.8) -- (0.5,0.2);
			\draw[->] (1,0.2) -- (3,1.8);
			\draw[->, red] (1, 3.8) -- (3,2.2);
		\end{tikzpicture}
	\  \cong \
	\mathrm{Tot}_{\begin{tikzpicture}[anchorbase]
			\filldraw[gray] (0,0) circle (2pt);
	\end{tikzpicture}}\left( 
	\begin{tikzpicture}[anchorbase]
		\filldraw[black] (0,0,0) circle (2pt);
		\filldraw[black] (0.5,0,0) circle (2pt);
		\filldraw[black] (1,0,0) circle (2pt);
		\filldraw[black] (1.5,0,0) circle (2pt);
		\filldraw[black] (1.5,0,-4) circle (2pt);
		\filldraw[black] (1,0,-4) circle (2pt);
		\filldraw[black] (0.5,0,-4) circle (2pt);
		\filldraw[black] (0,0,-4) circle (2pt);
		\draw[ultra thick] (0,0,-4) -- (0.5,0,-4);
		\filldraw[black] (4.5,0,0) circle (2pt);
		\filldraw[black] (4,0,0) circle (2pt);
		\filldraw[black] (3.5,0,0) circle (2pt);
		\filldraw[black] (3,0,0) circle (2pt);
		\draw[ultra thick] (4,0,0) -- (4.5,0,0);
		\filldraw[black] (4.5,0,-4) circle (2pt);
		\filldraw[black] (4,0,-4) circle (2pt);
		\filldraw[black] (3.5,0,-4) circle (2pt);
		\filldraw[black] (3,0,-4) circle (2pt);
		\draw[ultra thick] (4,0,-4) -- (4.5,0,-4);
		\draw[ultra thick] (3,0,-4) -- (3.5,0,-4);
		\filldraw[black] (0,3,0) circle (2pt);
		\filldraw[black] (0.5,3,0) circle (2pt);
		\filldraw[black] (1,3,0) circle (2pt);
		\filldraw[black] (1.5,3,0) circle (2pt);
		\draw[ultra thick] (0.5,3,0) -- (1,3,0);
		\draw[->] (0.75,0.1,0) -- (0.75,2.8,0);
		\draw[->] (1.7,0,0) -- (2.8,0,0);
		\draw[->] (1.7,0,-4) -- (2.8,0,-4);
		\draw[->] (0.75,0,-0.1) -- (0.75, 0,-3.5);
		\draw[->] (3.75,0,-0.1) -- (3.75, 0,-3.5);
		\filldraw[gray] (1.5,3,-4) circle (2pt);
		\filldraw[gray] (1,3,-4) circle (2pt);
		\filldraw[gray] (0.5,3,-4) circle (2pt);
		\filldraw[gray] (0,3,-4) circle (2pt);
		\draw[ultra thick, gray] (0,3,-4) -- (1,3,-4);
		\filldraw[gray] (4.5,3,0) circle (2pt);
		\filldraw[gray] (4,3,0) circle (2pt);
		\filldraw[gray] (3.5,3,0) circle (2pt);
		\filldraw[gray] (3,3,0) circle (2pt);
		\draw[ultra thick, gray] (3.5,3,0) -- (4.5,3,0);
		\filldraw[gray] (4.5,3,-4) circle (2pt);
		\filldraw[gray] (4,3,-4) circle (2pt);
		\filldraw[gray] (3.5,3,-4) circle (2pt);
		\filldraw[gray] (3,3,-4) circle (2pt);
		\draw[ultra thick, gray] (4.5,3,-4) -- (3,3,-4);
		\draw[->, gray, dotted] (0.75,0.1,-4) -- (0.75,2.8,-4);
		\draw[->, gray, dotted] (3.75,0.1,0) -- (3.75,2.8,0);
		\draw[->, gray, dotted] (3.75,0.1,-4) -- (3.75,2.8,-4);
		\draw[->, dotted, gray] (1.7,3,0) -- (2.8,3,0);
		\draw[->, gray, dotted] (1.7,3,-4) -- (2.8,3,-4);
		\draw[->, gray, dotted] (0.75,3,-0.1) -- (0.75, 3,-3.5);
		\draw[->, gray, dotted] (3.75,3,-0.1) -- (3.75, 3,-3.5);
	\end{tikzpicture}
	\right)
		\end{gather*}
		\caption{The complex $\overline{C}_4$ over the additive closure of $\Gamma_4$. Red arrows indicate morphisms that carry the sign $-1$. On the right, we see the poset structure of $\Gamma_4$ embedded in the three dimensional cube highlighting the complement in gray. Quotienting by the gray part and totalizing (in either order) yields an isomorphic complex.} 
		\label{fig:gammacomplex}
	\end{figure}
\end{example}
\begin{lemma}
		There is a unique functor $\Gamma_n \to \TL$ that sends a composition $s = (s_1, s_2, \dots, s_k) \in \Gamma_n$ to the number of ones $r_s$ in it (as an object of the Temperley--Lieb category) and acts on morphisms by
		\[ \Hom(s,s'_i) \ni \ast \mapsto \tlcapids{a}{b} \]
		whenever $s'_i=(s_1, \dots, s_{i-1}, s_i + s_{i+1}, s_{i+2}, \dots, s_k ) \in \Gamma_n$ for $1 \leq i \leq k-1 $ as above. Here, $a$ (resp.\ $b$) denotes the number of ones in $s'_i$ left (resp.\ right) of $s_i + s_{i+1}$. In particular, $a+b = r_s -2$.
\end{lemma}
\begin{proof}
	Note that $\Gamma_n$ is transitively and reflexively generated by relations $s_i \leq s'_i$ when $s'_i \in \Gamma_n$. 
	We use this to define the functor on general morphisms and have to check that every square in $\Gamma_n$ is mapped to a commutative square in $\TL$.
	Let $s \in \Gamma_n$. If $r_s < 4$, there is no square in $\Gamma_n$ starting at $s$, so suppose $s$ is of the form $s = (\dots 1,1, \dots, 1, 1, \dots)$ with $4 \leq r_s$ ones.
	Now we have the commutative square
		\begin{gather*}
	\begin{tikzcd}[ampersand replacement=\&]
	s = (\dots, 1,1, \dots, 1, 1, \dots)	\ar[d] \ar[r]  \&  (\dots, 2, \dots, 1, 1, \dots) \ar[d] \\
		(\dots, 1,1, \dots, 2, \dots) \ar[r] \&   (\dots, 2, \dots, 2, \dots)\\
	\end{tikzcd} 
	\quad \text{that is sent to} \quad
	\begin{tikzcd}[ampersand replacement=\&]
		r_s	\ar[d] \ar[r]  \&  r_s-2 \ar[d] \\
		r_s-2 \ar[r] \&   r_s-4\\
	\end{tikzcd}
\end{gather*}
	with the top and bottom composites being
	\begin{gather*}
		\begin{tikzpicture}[anchorbase, scale=0.6]
			\draw[usual] (3.3,0) to (3.3,1) to[out=90, in=180] (3.8,1.5) to[out=0, in=90] (4.3,1) to (4.3,0);
			\draw[usual] (1,0) to[out=90, in=180] (1.5,0.5) to[out=0, in=90]  (2,0);
			\draw[usual] (0,0) to (0,2);
			\node at (0.4,1) {\tiny $\dots$};
			\draw[usual] (0.7,0) to (0.7,2);
			\draw[usual] (2.3,0) to (2.3,2);
			\node at (2.7,1) {\tiny $\dots$};
			\draw[usual] (3,0) to (3,2);
			\draw[usual] (0.7,0) to (0.7,2);
			\draw[usual] (4.6,0) to (4.6,2);
			\node at (5,1) {\tiny $\dots$};
			\draw[usual] (5.3,0) to (5.3,2);
		\end{tikzpicture}
			\quad \text{and} \quad
		\begin{tikzpicture}[anchorbase, scale=0.6]
			\draw[usual] (0,0) to (0,2);
			\node at (0.4,1) {\tiny $\dots$};
			\draw[usual] (0.7,0) to (0.7,2);
			\draw[usual] (1,0) to (1,1) to[out=90, in=180] (1.5,1.5) to[out=0, in=90] (2,1) to (2,0);
			\draw[usual] (2.3,0) to (2.3,2);
			\node at (2.7,1){\tiny $\dots$};
			\draw[usual] (3,0) to (3,2);
			\draw[usual] (0.7,0) to (0.7,2);
			\draw[usual] (3.3,0) to[out=90, in=180] (3.8,0.5) to[out=0, in=90]  (4.3,0);
			\draw[usual] (4.6,0) to (4.6,2);
			\node at (5,1) {\tiny $\dots$};
			\draw[usual] (5.3,0) to (5.3,2);
		\end{tikzpicture}.
	\end{gather*}
	These are isotopic diagrams and therefore equal morphisms in $\TL$. Hence, the right square is commutative. Since the $1$-skeleton of a hypercube commutes if and only if every square commutes, this assignment is indeed functorial.
\end{proof}

We use this functor to construct the complex $C_n$ from $\overline{C}_n$ as follows:
\begin{construction}
	\label{cons:complex}
	Consider the functor $\Gamma_n \to \TL$ from the previous lemma. Extend it both additively and to chain complexes over $\Mat\left(\Kar\left(\TL\right)\right)$. Finally, denote by $C_n$ the image of $\overline{C}_n$ under this functor.
	The resulting complex $C_n$ (for even $n$, not showing the precise differential) is of the form
	\[
	\begin{tikzcd}
		\emptyset \arrow[r] & 	\underline{n} \arrow[r] & (n-2)^{\oplus \binom{n-1}{1}} \arrow[r] & \dots \arrow[r] & (n-2k)^{\oplus \binom{n-k}{k}} \arrow[r] & \dots \arrow[r] & 0 \arrow[r] & \emptyset
	\end{tikzcd}
	\]
	where homological degree zero is underlined. 
	Here, we denote the zero object by $\emptyset$ to avoid confusion with the monoidal unit $0 \in \TL$.
	Khovanov \cite{khovanov2005categorifications} originally introduced this complex in terms of representations of the Lie algebra $\mathfrak{sl}_2$. Applying the automorphism $(-)^\dynkin$ of $\TL$ to $C_n$ yields an isomorphic complex by (signed) reordering of components. This automorphism of $C_n$ can be understood on the level of $\overline{C}_n$ by the corresponding automorphism of $\overline{C}_n$ given (up to sign) by mapping a composition with entries $(s_1, \dots, s_k)$ to the composition $(s_k, \dots, s_1)$.
\end{construction}

\begin{example}
	Applying \autoref{cons:complex} to the complex $\overline{C}_4$ from \autoref{ex:cbar}, we obtain the complex $C_4$ as shown on the left of \autoref{fig:tlcomplex}. The right-hand side highlights the description as a cone that will become important in the proof of \autoref{lem:hte}.
	\begin{figure}[H]
	\centering
	\begin{gather*}
		\left(
\begin{tikzpicture}[baseline=-0.5ex]
	\node at (-3,0) (A) {$\underline{4}$};
	\node at (0,0) (B) {$2^{\oplus 3}$};
	\node at (3,0) (C) {$0$};
	\draw[->] (A) -- node[midway, above] {$\begin{pmatrix}
			\begin{tikzpicture}[scale= 0.25]
				\draw[usual,-] (1,0) to[out=90, in=180] (1.5,0.5) to[out=0, in=90] (2,0);
				\draw[usual,-] (2.5, 0) to (2.5,1);
				\draw[usual,-] (3, 0) to (3,1);
			\end{tikzpicture}\\
			\begin{tikzpicture}[scale= 0.25]
				\draw[usual,-] (0.5, 0) to (0.5,1);
				\draw[usual,-] (1,0) to[out=90, in=180] (1.5,0.5) to[out=0, in=90] (2,0);
				\draw[usual,-] (2.5, 0) to (2.5,1);
			\end{tikzpicture}\\
			\begin{tikzpicture}[scale= 0.25]
				\draw[usual,-] (0, 0) to (0,1);
				\draw[usual,-] (0.5, 0) to (0.5,1);
				\draw[usual,-] (1,0) to[out=90, in=180] (1.5,0.5) to[out=0, in=90] (2,0);
			\end{tikzpicture}
		\end{pmatrix}$} (B);
	\draw[->] (B) -- (C) node[midway, above] {$\begin{pmatrix}
			-\begin{tikzpicture}[scale= 0.25]
				\draw[usual,-] (1,0) to[out=90, in=180] (1.5,0.5) to[out=0, in=90] (2,0);
			\end{tikzpicture}
			&
			0
			&
			\begin{tikzpicture}[scale= 0.25]
				\draw[usual,-] (1,0) to[out=90, in=180] (1.5,0.5) to[out=0, in=90] (2,0);
			\end{tikzpicture}
		\end{pmatrix}$};
\end{tikzpicture}
\right)
\quad = \quad \left(
\begin{tikzpicture}[anchorbase]
				\node at (-1.7,2) {$4$};
				\node at (0.75,0) {$2$};
				\node at (0.75,2) {$2$};
				\node at (0.75,4) {$2$};
				\node at (3.2,2) {$0$};
				\draw  (0.75,3) node {$\bigoplus$};
				\draw  (0.75,1) node {$\bigoplus$};
				\draw[->] (-1.3,2) -- (0.5,2) node[midway, above left](A) {};
				\draw[->] (-1.5,2.2) -- (0.5, 3.8) node[midway, left](B) {};
				\draw[->, blue] (-1.5,1.8) -- (0.5,0.2);
				\draw[->] (1,0.2) -- (3,1.8);
				\draw[->, blue] (1, 3.8) -- (3,2.2);
				\begin{scope}[shift={(-1.125,3.2)}, scale= 0.25]
						\draw[usual] (1,0) to[out=90, in=180] (1.5,0.5) to[out=0, in=90] (2,0);
						\draw[usual] (2.5, 0) to (2.5,1);
						\draw[usual] (3, 0) to (3,1);
					\end{scope}
				\begin{scope}[shift={(-0.5,2.2)}, scale= 0.25]
						\draw[usual] (0.5, 0) to (0.5,1);
						\draw[usual] (1,0) to[out=90, in=180] (1.5,0.5) to[out=0, in=90] (2,0);
						\draw[usual] (2.5, 0) to (2.5,1);
					\end{scope}
				\begin{scope}[shift={(-1.125,0.8)}, scale= 0.25]
						\draw[usual] (0, 0) to (0,1);
						\draw[usual] (0.5, 0) to (0.5,1);
						\draw[usual] (1,0) to[out=90, in=180] (1.5,0.5) to[out=0, in=90] (2,0);
					\end{scope}
				\begin{scope}[shift={(1.9,3.2)}, scale= 0.25]
						\draw[usual] (1,0) to[out=90, in=180] (1.5,0.5) to[out=0, in=90] (2,0);
					\end{scope}
				\begin{scope}[shift={(1.7,0.8)}, scale= 0.25]
						\draw[usual] (1,0) to[out=90, in=180] (1.5,0.5) to[out=0, in=90] (2,0);
					\end{scope}
				\node at (2,3.25) {$-$};
			\end{tikzpicture}
			\right)
\end{gather*}
	\caption{The complex $C_4$ over $\Mat(\Kar(\TL))$: On the left-hand side as it was constructed in \autoref{cons:complex}. On the right-hand side, it can be seen as a cone on a chain map along the blue arrows (where the signs are introduced by our convention of cones below).}
	\label{fig:tlcomplex}
\end{figure}
\end{example}

\begin{lemma}
	\label{lem:hte}
	Consider the complex $C_n$ as constructed in \autoref{cons:complex} over a field of characteristic zero. As a bounded chain complex over $\Mat\left(\Kar\left(\TL\right)\right)$ it is homotopy equivalent to 
	\[
	\begin{tikzcd}[ampersand replacement=\&]
		\emptyset \arrow[r] \& 	\underline{\karjw} \arrow[r] \&  \emptyset
	\end{tikzcd}
	\]
	where (from now on) $\karjw$ denotes the object in $\Kar\left(\TL\right)$ corresponding to the Jones--Wenzl idempotent $p_n$.
\end{lemma}
A proof of this statement using the concept of Chebyshev systems appears in \cite[Section 4]{beliakova2025unificationcoloredannularsl2}. We include a direct proof for the readers' convenience. 
\begin{proof}
	We want to argue inductively. The cases $n = 0,1$ are clear since both complexes are identical. For general $n$ we argue as follows:
	
	As already pointed out in \cite{205649, khovanov2005categorifications}, there are
	two kinds of compositions contributing to the direct sum in every
	homological degree of the complex $\overline{C}_n$: Compositions ending in 2 and
	compositions that end in a 1. Since the differential never splits a 2 into
	$1+1$, the former is a subcomplex isomorphic to $\overline{C}_{n-2}$ shifted once to
	the right that will denoted by $\overline{C}_{n-2} [-1]$ and the latter class of
	compositions can be obtained as the quotient complex by $\overline{C}_{n-2} [-1]$ and
	is isomorphic to $\overline{C}_{n-1} \otimes 1$. 

	This filtration arises as a cone of the two complexes, such that we can write
	\begin{equation*}
		\overline{C}_n = \Cone(\overline{C}_{n-1} \otimes 1 \xrightarrow{\overline{\delta}} \overline{C}_{n-2}).
	\end{equation*} 
	where the non-zero components of $\overline{\delta}$ are given by addition of the last two ones in a composition. In order to match the signs on both sides, we adapt the following convention for a cone on a chain map $f\colon (A,d_A) \to (B,d_B)$ between complexes. In homological degree $k$, we place the chain object $A_k \oplus B_{k-1}$ and use the differential
	\begin{align*}
	d_{\Cone(f);k}\colon A_k \oplus B_{k-1} \to  A_{k+1} \oplus B_{k} \quad \text{given by the matrix} \quad 
	\begin{pmatrix}
		d_{A,k} & 0 \\
		(-1)^k f_k & d_{B,k-1}
	\end{pmatrix}
	\end{align*}
	where the subscripts indicate the sources of the maps.
	This description as a cone is preserved by the functor mapping $\overline{C}_n$ to $C_n$ where $\delta$ is the image of $\overline{\delta}$ under the functor used in \autoref{cons:complex}:
	\begin{gather}
		\label{eq:cone}
		C_n = \Cone(C_{n-1} \otimes 1 \xrightarrow{\delta} C_{n-2})
		\quad
		\text{and has components}
		\quad 
		\delta_k = 
		\begin{tikzpicture}[anchorbase, scale=0.45]
			\draw[usual] (-0.5, -1) to (-0.5,1);
			\draw (0.05,0) node {$\dots$};
			\draw[usual] (0.5, -1) to (0.5,1);
			\draw[usual] (1,-1) to[out=90, in=180] (1.5,-0.5) to[out=0, in=90] (2,-1);
		\end{tikzpicture}
	\end{gather}
	with $n-2k-2$ parallel strands.
	By induction, $C_{n-1}$ and $C_{n-2}$ are homotopy equivalent to the one-term complexes having $\karjw[n-1]$ and $\karjw[n-2]$ in degree $0$. Now we can replace these two terms in the cone \eqref{eq:cone} to obtain a homotopy equivalence 
	
	\[C_n = \Cone(C_{n-1} \otimes 1 \xrightarrow{\delta} C_{n-2}) \simeq \Cone(\karjw[n-1] \otimes 1 \xrightarrow{\beta} \karjw[n-2]) 
	\quad \text{with}\quad 
	\beta = \
	\begin{tikzpicture}[anchorbase,scale=.3]
		\draw[usual] (0.7,-1.5) to (0.7,1.5);
		\draw[usual] (-1.2,-1.5) to (-1.2,1.5);
		\draw[usual] (1.2,-1.5) to (1.2,.5) to [in=180, out = 90] (1.7,1) to[out = 0, in = 90] (2.2, .5) to (2.2,-1.5);
		\draw (-0.15,-1) node {$\dots$};
		\draw (-0.15,1) node {$\dots$};
		\draw[JW, fill=white] (-1.5,-.5) rectangle (1.5,.5);
		\node at (0,0) {\small $n-1$};
	\end{tikzpicture}
	\]
	where $\begin{tikzpicture}[anchorbase,scale=.3]
		\draw[JW, fill=white] (-1.5,-.5) rectangle (1.5,.5);
		\node at (0,0) {\small $n-1$};
	\end{tikzpicture}$ denotes the $(n-1)$-th Jones--Wenzl projector. (The replacement can be seen as an easy instance of homological perturbation, see Markl \cite{markl2001ideal} and also Hogancamp \cite{hogancamp2020homological}.) 
	
	Using that $\karjw[n-1] \otimes \id_1 \cong \karjw \oplus \karjw[n-2]$ in $\Mat\left(\Kar\left(\TL\right)\right)$ via the pair of mutually inverse maps
	\begin{gather*}
		r = \begin{pmatrix}
			\begin{tikzpicture}[anchorbase,scale=.3]
				\draw[usual] (1.3,-1.5) to (1.3,1.5);
				\draw[usual] (-1.3,-1.5) to (-1.3,1.5);
				\draw[JW, fill=white] (-1.5,-.5) rectangle (1.5,.5);
				\node at (0,0) {\small $n$};
				\draw (0.1,-1) node {$\dots$};
				\draw (0.1,1) node {$\dots$};
			\end{tikzpicture}\\
			\begin{tikzpicture}[anchorbase,scale=.3]
				\draw[usual] (0.7,-1.5) to (0.7,1.5);
				\draw[usual] (-1.2,-1.5) to (-1.2,1.5);
				\draw[usual] (1.2,-1.5) to (1.2,.5) to [in=180, out = 90] (1.7,1) to[out = 0, in = 90] (2.2, .5) to (2.2,-1.5);
				\draw (-0.15,-1) node {$\dots$};
				\draw (-0.15,1) node {$\dots$};
				\draw[JW, fill=white] (-1.5,-.5) rectangle (1.5,.5);
				\node at (0,0) {\small $n-1$};
			\end{tikzpicture}
		\end{pmatrix}
		\quad
		\text{and}
		\quad
		s=\begin{pmatrix}
			\begin{tikzpicture}[anchorbase,scale=.3]
				\draw[usual] (1.3,-1.5) to (1.3,1.5);
				\draw[usual] (-1.3,-1.5) to (-1.3,1.5);
				\draw[JW, fill=white] (-1.5,-.5) rectangle (1.5,.5);
				\node at (0,0) {\small $n$};
				\draw (0.1,-1) node {$\dots$};
				\draw (0.1,1) node {$\dots$};
			\end{tikzpicture}&
			\frac{[n-1]}{[n]}
			\begin{tikzpicture}[anchorbase,scale=.3]
				\draw[usual] (0.7,-1.5) to (0.7,1.5);
				\draw[usual] (-1.2,-1.5) to (-1.2,1.5);
				\draw[usual] (1.2,1.5) to (1.2,-.5) to [in=180, out = 270] (1.7,-1) to[out = 0, in = 270] (2.2, -.5) to (2.2,1.5);
				\draw (-0.15,-1) node {$\dots$};
				\draw (-0.15,1) node {$\dots$};
				\draw[JW, fill=white] (-1.5,-.5) rectangle (1.5,.5);
				\node at (0,0) {\small $n-1$};
			\end{tikzpicture}
		\end{pmatrix}
	\end{gather*}
where the identities 
\[ r \circ s = \begin{pmatrix}
	\begin{tikzpicture}[anchorbase,scale=.3]
		\draw[usual] (1.3,-1.5) to (1.3,1.5);
		\draw[usual] (-1.3,-1.5) to (-1.3,1.5);
		\draw[JW, fill=white] (-1.5,-.5) rectangle (1.5,.5);
		\node at (0,0) {\small $n$};
		\draw (0.1,-1) node {$\dots$};
		\draw (0.1,1) node {$\dots$};
	\end{tikzpicture} & 0 \\
	0 & \begin{tikzpicture}[anchorbase,scale=.3]
		\draw[usual] (1.3,-1.5) to (1.3,1.5);
		\draw[usual] (-1.3,-1.5) to (-1.3,1.5);
		\draw[JW, fill=white] (-1.5,-.5) rectangle (1.5,.5);
		\node at (0,0) {\small $n-2$};
		\draw (0.1,-1) node {$\dots$};
		\draw (0.1,1) node {$\dots$};
	\end{tikzpicture}&
\end{pmatrix} = \id_{\karjw \oplus \karjw[n-2]} \quad \text{and} \quad s \circ r = \begin{tikzpicture}[anchorbase,scale=.33]
\draw[JW] (-1.5,0) rectangle (1.5,2);
\node at (0,.9) {$n-1$};
\draw[usual] (2,0) to (2,2);
\end{tikzpicture} = \id_{\karjw[n-1] \otimes 1} \]
follow directly from the properties of Jones--Wenzl projectors in \autoref{def:JW}. This shows that $\beta$ is precisely the projection on the $\karjw[n-2]$ summand and therefore restricts to an isomorphism (even the identity) between these two components. Therefore, we can apply Gaussian elimination as in \cite{barnatan2007fastkhovanovhomologycomputations} to contract them and finally obtain the desired statement.
\end{proof}

\begin{remark}
	Totalizations of Fibonacci cubes were already studied in \cite{dyckerhoff2023nsphericalfunctorscategorificationeulers} with connections to Jones--Wenzl projectors observed in \cite{Coulembier2024Nsphericalfunctorsandtensorcategories}. The full cubes of compositions play an analogous role for higher rank type A colored link homology and, relatedly, perverse schobers of Coxeter type A \cite{dyckerhoff2025perverseschoberscoxetertype}.
\end{remark}
\subsection{Extension to links}
\label{subsec:extension_to_links}

With these constructions in place, we are ready to cover the case of knots. More generally, we would like to extend them to links and work in that setting from now on: Let $L$ be a framed, oriented link with $r$ ordered components $K_1, \dots, K_r$. Given such a link, we can form its $\mathbf{n} = (n_1, \dots, n_r)$-cable $L^\mathbf{n}$ having $n_i$ parallel copies of $K_i$. 
For each of the components $K_i$, we orient the parallel copies in an alternating way as before and extend the braiding cobordisms defined for knots earlier to the link $L$ by locally braiding the parallel copies of $K_i$ and extend by the identity cobordism on the rest of $L^\mathbf{n}$. As a consequence, the braiding cobordisms corresponding to different components commute with each other. By functoriality, this gives rise to commuting actions of the braid groups $\mathrm{Br}_{n_i}$ on $\Kh(L^\mathbf{n})$. The same argument due to \cite{Grigsby_2017} as for knots shows that each of these actions factors through $S_{n_i}$, respectively, yielding in total an action of $S_{\mathbf{n}} := S_{n_1} \times \dots \times S_{n_r}$ on $\Kh(L^\mathbf{n})$. Similarly, $\Kh(L^\mathbf{n})$ is a representation of $\TL_\mathbf{n} := \TL_\mathbf{n}(2) := \TL_{n_1}(2) \times \dots \times \TL_{n_r}(2)$.

As in the case of knots, where we used this to define a functor $\Kh_K$ for a fixed framed, oriented knot $K$, we now define a functor $\Kh_L$ for a fixed framed, oriented link $L$ as above by
\begin{equation*} 
	\Kh_L^\commR\colon \prod_{i = 1}^{r} \TL \to \Mod^{\Z \times \Z} \\
\end{equation*}
that sends the object $\mathbf{n} = (n_1, \dots, n_r)$ to $\Kh(L^\mathbf{n})$ and acts on morphisms component-wise by cobordism maps similar to the ones constructed for $\Kh_K$.

Finally, we define the chain complexes $C_\mathbf{n}$ by
\[ C_\mathbf{n} := \bigotimes_{i = 1}^{r} (0,\dots,0, C_{n_i},0, \dots, 0) \] 
as a chain complex over $\prod_{i = 1}^{r} \Mat(\Kar(\TL)) =: \Mat(\Kar(\TL))^{\times r}$. Here, $C_{n_i}$ is the complex as in \autoref{cons:complex} in the $i$-th position and $0$ is the monoidal unit in $\TL$. Over a field of characteristic zero, there is a homotopy equivalence $ C_\mathbf{n}\simeq \prod_{i = 1}^{r} \karjw[n_i] =: \karjw[\mathbf{n}]$.
The homotopy equivalence data is given by applying \autoref{lem:hte} to every tensor factor separately. We will also use the dual complex $C_\mathbf{n}^*$ obtained by negating the homological degree and dualizing all chain objects and differentials using the duality in $\TL$.

Now we are ready to construct all versions of colored Khovanov homology that we want to consider: 

\begin{definition}
	\label{def:coloredHomology}
 Let $\commR$ be a commutative ring. We define the following variations of $\mathbf{n}$-colored Khovanov homology $\cKh{n}(L)$ of $L$ over $\commR$:
	\begin{center}
		\begin{tabular}{ r l}
			(\invlabel) & The $S_\mathbf{n}$ invariants $\Kh\left(L^\mathbf{n}\right)^{S_\mathbf{n}}$. \\ 
			(\coinvlabel) &  The $S_\mathbf{n}$ co-invariants $\Kh\left(L^\mathbf{n}\right)_{S_\mathbf{n}} = \Kh\left(L^\mathbf{n}\right)/ \linspan \{ \sigma.v - v \mid v\in \Kh\left(L^\mathbf{n}\right), \sigma \in S_\mathbf{n} \}$.  \\  
			(\cohomlabel) & The cohomology of $\Kh_L(C_\mathbf{n})$ for $C_\mathbf{n}$ as generalized to links from \autoref{cons:complex} above.\\
			(\homlabel) &  The cohomology $\Kh_L(C_\mathbf{n}^*)$ where $C_\mathbf{n}^*$ is the dual of $C_\mathbf{n}$.\\
			(\kerlabel) & The cohomology of $\Kh_L(C_\mathbf{n})$ in degree $0$, i.e.~the kernel of the first (non-zero) differential.  \\
			& This is a map $\Kh(L^\mathbf{n}) \to \bigoplus_{i = 1}^r \Kh(L^{\mathbf{n}-2\mathbf{e_i}})^{\oplus n_i-1}$ where $r$ is the number of link components  \\ 
			&of $L$ and $\mathbf{e_i}$ is the $i$-th $r$-component unit vector. Its components are given by the cobordism \\
			& maps of annuli as introduced earlier. \\
			(\cokerlabel) & The cohomology of $\Kh_L(C_\mathbf{n}^*)$ in degree $0$, i.e.~the cokernel of the last (non-zero) differential.  \\
			& This is a similar map  $\bigoplus_{i = 1}^r \Kh(L^{\mathbf{n}-2\mathbf{e_i}})^{\oplus n_i-1} \to \Kh(L^\mathbf{n})$ with components given by cobordism \\
			& maps of the same embedded annuli read in the opposite direction.
		\end{tabular}
	\end{center}
	For the following definitions, we work over a field $\F$ of characteristic zero:
	\begin{center}
		\begin{tabular}{ r l}
		(\imlabel) & The image $\Kh_L(\karjw[\mathbf{n}])$ of the Jones--Wenzl projector $\karjw[\mathbf{n}] \in \Kar(\TL)$ under the functor $\Kh_L$. \\ 
		& This can also be thought of as the image  $\im \Kh_L(p_\mathbf{n}) \subseteq \Kh\left( L^{\mathbf{n}}\right)$ of the endomorphism $ \Kh_L(p_\mathbf{n}) $ \\ 
		&if we consider the idempotent $p_\mathbf{n} \in \Hom_{ \left(\TL^\F \right)^{\times r}}(\mathbf{n},\mathbf{n})$ as mentioned earlier. \\
		(\coimlabel) &  The quotient $\Kh\left( L^{\mathbf{n}}\right)/\ker \Kh_L(p_\mathbf{n}) = \Kh\left( L^{\mathbf{n}}\right)/\im \Kh_K(\id - p_\mathbf{n})$  of $\Kh\left( L^{\mathbf{n}}\right)$ by the image \\ & of the complementary idempotent to $p_\mathbf{n}$, using the second interpretation of $p_\mathbf{n}$ above.   
		\end{tabular}
	\end{center}
	
\end{definition}

\begin{remark}
	A (probably non-complete) list of appearances of definitions above in the literature include:
	\begin{itemize}
		\item (\cohomlabel) is Khovanov's original definition of the categorification of the colored Jones polynomial from \cite{khovanov2005categorifications} over a field. The relations to (\homlabel), (\kerlabel), (\cokerlabel) and (\invlabel) are also mentioned and isomorphisms between them are conjectured when working over a field of characteristic zero.
		\item Beliakova--Wehrli \cite{Beliakova_Wehrli_2008} define the chain complexes $C_\mathbf{n}$ (denoted $\Gamma_\mathbf{n}$ there) and show that Khovanov's construction can be defined over the integers.
		\item The version (\coinvlabel) naturally appears in Manolescu--Neithalath's 2-handle formula for the skein lasagna modules of 4 manifolds presentable by attaching 2-handles to the 4-ball~\cite{manolescu2022skein}.  Note that they work over $\Z$ and not over a field of characteristic zero.
		\item  Grigsby--Licata--Wehrli \cite{Grigsby_2017} claim (\invlabel) $\cong$ (\cohomlabel) over a field of characteristic zero. Later, a proof of (\imlabel) $\cong$ (\cohomlabel) in terms of Chebyshev systems appeared in Beliakova--Hogancamp--Putyra--Wehrli \cite{beliakova2025unificationcoloredannularsl2}.
		\item Hogancamp--Rose--Wedrich \cite{Hogancamp2025-ol} argue that over a field of characteristic zero (\invlabel) $\cong$ (\cohomlabel) $\cong$ (\imlabel) for knots and show (\invlabel) $\cong$ (\coinvlabel) in the context of the 2-handle formula for skein lasagna modules. The proof given in this note is the generalization of their proof to links with some details filled in.
		
	\end{itemize}
\end{remark}

\begin{proposition}
	\label{prop:kerandinv}
		Let $\commR$ be a commutative ring where 2 is not a zero divisor. Then, (\kerlabel) $\cong$ (\invlabel) over $\commR$. Further, if $2$ is a unit in $\commR$, we additionally have (\cokerlabel) $\cong$ (\coinvlabel) over $\commR$.
\end{proposition}
\begin{proof}
	We prove both statements separately.
	\paragraph{(\kerlabel) $\cong$ (\invlabel):}
	We proof the statement for $n$-cables of a knot $K$. The general case for multi-component links follows by taking intersections over the corresponding subspaces of $\Kh(L^\mathbf{n})$ for the (commuting) actions on the cables of each link component of $L$.
	
	Recall that the crossing in $\TL$ of the $i+1$-th and $i+2$-th strand is given by
		\begin{equation} 
			\label{eq:tlcrossingati}
			\tlcrossingids{i}{} = \id - \tldensecupcapids{i}{}.
		\end{equation} Therefore, for $v \in \Kh(K^n)$, we have the implication 
		\[ \Kh_K\left(\tlcapids{i}{}\right)(v) = 0 \implies \Kh_K\left(\tlcrossingids{i}{}\right)(v) = v \] which implies (\kerlabel) $\subseteq$ (\invlabel). To show the opposite inclusion, we post-compose equation \eqref{eq:tlcrossingati}  with $\tlcapids{i}{}$ to obtain
	\[ \begin{tikzpicture}[anchorbase,scale=0.15, tinynodes]
			\draw[usual] (-1.5,-1) to (-1.5,0.3) node[left,xshift=2pt]{$i$} to (-1.5,2.5);
			\draw[usual] (-1,-1) to [out=90,in=270] (1,1) to [out=90,in=0] (0,1.75) to [out=180,in=90](-1,1);
			\draw[usual]  (1,-1)  to  [out=90,in=270] (-1,1);
			\draw[usual] (1.5,-1) to (1.5,2.5);
	\end{tikzpicture}
	\;\;
	=
	- \
	\begin{tikzpicture}[anchorbase, scale=0.15,tinynodes]
		\draw[usual] (-1.5,0) to (-1.5,0.93) node[left,xshift=2pt]{$i$} to (-1.5,2.5);
		\draw[usual] (-1,0) to[out = 90, in = 180] (0,1) to [out = 0, in = 90]  (1,0);
		\draw[usual] (1.5,0)  to (1.5,2.5);
	\end{tikzpicture}
	 \]
	in general, but on $S_n$-invariants, $\Kh_K\left( \tlcrossingids{i}{} \right)$ acts trivially and therefore both $\Kh_K \left( \tlcapids{i}{} \right)$ and $-\Kh_K\left(\tlcapids{i}{}\right)$ act by the same morphism $\Kh(K^n)^{S_n} \to \commR$. If $2$ is not a zero divisor in $\commR$, this must be the zero morphism.
	\paragraph{(\cokerlabel) $\cong$ (\coinvlabel):}
	Similar to the previous case, we argue locally for an $n$-cable of a link component $K$ and compare the spaces by which we take the quotient. The general statement follows by taking the respective inner sums of these subspaces for each link component.
	
	Consider the subspaces of $\Kh\left(K^n\right)$
	\[
		S =  \linspan \{ \sigma.v - v \mid v\in \Kh\left(K^n\right), \sigma \in S_n \} \quad \text{and} \quad
		U = \sum_{i+j = n-2} \im \Kh_K \left(\tlcupids{i}{j}\right) 
	\]
	and note that it suffices to restrict to generators (i.e.\ crossings in all positions) in the definition of $S$. Simplifying then yields
	\[ S = \sum_{i+j = n-2} \im \Kh_K \left(\tlcupcapids{i}{j}\right) \subseteq U. \]
	
	For the other inclusion, we compute
	\[ 2\cdot U = \sum_{i+j = n-2} \im 2 \Kh_K \left(\tlcupids{i}{j}\right) = \sum_{i+j = n-2} \im \Kh_K \left( \begin{tikzpicture}[anchorbase, tinynodes, scale=0.15]
	\draw[usual] (-1.5,-3.5) to (-1.5,-1.5) node[left,xshift=2pt]{$i$} to (-1.5,0.5);
	\draw[usual] (-1,0.5) to[out = 270, in = 180] (0,-0.5) to[out = 0, in = 270] (1,0.5);
	\draw[usual] (-1,-2.5) to[out = 90, in = 180] (0,-1.5) to[out = 0, in =90] (1,-2.5) to[out=270,in = 0] (0,-3.5) to[out=180, in=270] (-1,-2.5);
	\draw[usual] (1.5,-3.5) to (1.5,-1.5) node[right,xshift=-2pt]{$j$} to (1.5,0.5);
	\end{tikzpicture} \right)  \subseteq S \subseteq U \] where the second equality follows from the relation in $\TL$ and the first inclusion by pre-composing the maps in the respective summand with $\Kh_K(\tlcupids{i}{j})$. In particular, if $2$ is a unit in $\commR$, all of these agree and so do the quotients $\Kh(K^n)/S \cong \Kh(K^n)/U$.
\end{proof}

\subsection{Equivalence over a field of characteristic zero}
\label{subsec:proof}

We now want to prove the following proposition which is the extension of \autoref{prop:main} and \autoref{thm:introcableformula} to links: 
\begin{proposition}
	\label{prop:definitionsEquivalent}
	Let $\F$ be a field of characteristic zero. Then we have:
	\begin{enumerate}
		\item All versions of $\mathbf{n}$-colored Khovanov homology $\cKh{\mathbf{n}}(L)$ over $\F$ of a framed, oriented link $L$ in \autoref{def:coloredHomology} yield isomorphic bigraded vector spaces.
		\item 	The 2-variable Poincaré polynomial of $\mathbf{n}$-colored Khovanov homology can be computed from the Poincaré polynomials of the Khovanov homology of the cables of $L$ via: 
		\begin{equation}
			\label{eq:poincarecables}
			\Pol \left(\cKh{\mathbf{n}}(L)\right) = \sum_{\mathbf{k} = \mathbf{0}}^{\lfloor \frac{\mathbf{n}}{2} \rfloor} (-1)^{\abs{\mathbf{k}}} \binom{\mathbf{n}-\mathbf{k}}{\mathbf{k}} \Pol \left( \Kh \left( L^{\mathbf{n}-2\mathbf{k}} \right) \right) (t,q) \in \N_0[t, t^{-1}, q, q^{-1}] 
		\end{equation}
		where $\mathbf{k} = (k_1, \dots, k_r)$, $\abs{\mathbf{k}} = k_1 + \dots + k_r$ and
		\[ \binom{\mathbf{n}-\mathbf{k}}{\mathbf{k}} = \prod_{i=0}^{r} \binom{n_i - k_i}{k_i}. \]
	\end{enumerate}
\end{proposition}

\begin{remark}
	For the database \cite{coloredKh}, we (often) compute this Poincaré polynomial from the free part over $\Z$. By the universal coefficient theorem, this yields the same polynomial which therefore also does not depend on $\F$.
\end{remark}

\begin{remark}
	The coefficients in \eqref{eq:poincarecables} are governed by the Chebyshev polynomials of the second kind. Since those form a basis of the polynomial ring $\Z[X]$, the relation \eqref{eq:poincarecables} can be inverted to express the Poincaré polynomials of cables in terms of the Poincaré polynomials of colored homologies. For knots one obtains the formula in \autoref{cor:inverse} and for links a multilinear generalization.
\end{remark}

In order to phrase the next lemma, we will need the category $\text{Link}_{B^3}$. Objects are given by framed, oriented\footnote{Sometimes, similar categories also appear without the additional framing and orientation, but the framed and oriented setting is the most natural one for our purposes.} links smoothly embedded in the 3-ball $B^3$ and the morphisms $L_1 \to L_2$ are isotopy classes of framed, oriented compact surfaces $\Sigma$ smoothly embedded in $B^3 \times [0,1]$ connecting $L_1$ in $B^3\times \{0\}$ to $L_2$ in $B^3\times \{1\}$. Further, embedding 3-balls inside a larger 3-ball induces a symmetric monoidal\footnote{Strictly speaking: $\mathbb{E}_3$-monoidal.} structure on $\text{Link}_{B^3}$, which is rigid. The dual of a link is given by the mirror reverse $L^!$\footnote{Obtained on the level of diagrams by reflecting in the $y$-axis and reversing orientations.} and the counit and unit of the duality are given by the two ways to bend a cylinder $L \times [0,1]$ into embedded link cobordisms $\mathrm{ev}\colon L \sqcup L^! \to \emptyset$ and $\mathrm{coev}\colon \emptyset \to L^! \sqcup L$.

We want to factorize the cable functor $\Kh_L\colon \TL \to \Vect^{\Z \times \Z}$ as the composite of some geometric procedure $\cable_{L}$ and Khovanov homology. For this, we consider the quotient $\commR\mathrm{Link}_{/t}$ of the $\commR$-linearization of $\mathrm{Link}_{B^3}$ by the relation that every embedded torus evaluates to two. This relations ensures that for every framed, oriented link $L$ with $r$ ordered components, we have a (non-monoidal!) functor $\mathrm{cable}_L: \TL^{\times r} \to \commR\mathrm{Link}_{/t}$ that sends $\mathbf{n}$ to the $\mathbf{n}$-cable of $L$ and morphisms to the corresponding embedded annuli as described earlier. The (classes of the) same morphisms $\mathrm{ev}_L$ and $\mathrm{coev}_L$ as before exhibit $\commR\mathrm{Link}_{/t}$ as a rigid monoidal category, but those are not the images of the respective maps under the cabling functors since these are not monoidal.

\begin{remark}
	Instead of working with $\commR\mathrm{Link}_{/t}$, we could have extended the source of $\cable_{L}$ to $\commR$-linear combinations of crossingless tangle diagrams without the relation $\bigcirc = 2$, lift the complexes $C_\mathbf{n}$ to this setting, prove the next proposition and only pass to the quotient $\TL$ at the very end.
\end{remark}

\begin{lemma}
	\label{lem:mirrorandduality}
	Let $L$ be a framed, oriented link with $r$ ordered components. We have a natural isomorphism of functors
	\[ (-)^*_{\commR\mathrm{Link}_{/t}} \circ \cable_{L} \Longrightarrow \cable_{L^!}^{\mathrm{op}} \circ (-)^*_{\TL} \circ (-)^\dynkin_{\TL} \]
	with components given by isotopies of embeddings of $(L^!)^\mathbf{n}$. 
\end{lemma}
\begin{proof}
	First, note that the composition $ (-)^*_{\TL} \circ (-)^\dynkin_{\TL}$ corresponds to the reflection of Temperley--Lieb diagrams in a horizontal line. Naturality then follows by careful inspection of the resulting surfaces: Taking the position relative to the framing of components created by cups or annihilated by caps into account, left- and right-hand side give rise to isotopic cobordisms. We illustrate the action on morphisms for the example of \begin{tikzpicture}[scale=0.3, anchorbase]
		\draw[thick,line cap=round] (-0.5, -1) to[out=90, in = 180] (-0.25, -0.75) to[out=0, in = 90] (0,-1);
		\draw[thick,line cap=round] (0.5, -1) to[out=90, in = 270] (0,0);
	\end{tikzpicture}:
	
	Locally, in the neighborhood of a component knot $K$, the left-hand side produces a cobordism $K^! \to \left(K^! \right)^3$ creating two new components on the right. We indicate creation/annihilation of components by red/green rectangles and record the framing with a blue arrow on the first component (as seen by the framing) to obtain the following sequence of morphisms.
	\begin{gather*}
		\begin{tikzpicture}[scale=1, anchorbase]
			\draw[thick,line cap=round] (-0.5, -1) to[out=90, in = 180] (-0.25, -0.75) to[out=0, in = 90] (0,-1);
			\draw[thick,line cap=round] (0.5, -1) to[out=90, in = 270] (0,0);
		\end{tikzpicture}
		\xmapsto{\cable_K}
		\begin{tikzpicture}[scale=0.5, anchorbase]
			\fill[red!10] (-1,-1) rectangle (0,1);
			\draw[->] (-1,-1) -- (-1,1);
			\draw[->] (0,1) -- (0,-1);
			\draw[->] (1,-1) -- (1,1);
			\draw[->, blue] (-1,-0.5) -- (-0.5,-0.5);
		\end{tikzpicture}
		\xmapsto{(-)^*}
		\begin{tikzpicture}[scale=0.5, anchorbase]
			\fill[green!10] (0,-1) rectangle (1,1);
			\draw[->] (-1,1) -- (-1,-1);
			\draw[->] (0,-1) -- (0,1);
			\draw[->] (1,1) -- (1,-1);
			\draw[->, blue] (1,0.5) -- (0.5,0.5);
		\end{tikzpicture}
	\end{gather*}
	Similarly, for the right-hand side, the composition yields the following sequence of morphisms. The crucial detail here is that the new components have to be inserted \emph{before} the existing component (as seen by the framing).
	\begin{gather*}
		\begin{tikzpicture}[scale=1, anchorbase]
			\draw[thick,line cap=round] (-0.5, -1) to[out=90, in = 180] (-0.25, -0.75) to[out=0, in = 90] (0,-1);
			\draw[thick,line cap=round] (0.5, -1) to[out=90, in = 270] (0,0);
		\end{tikzpicture}
		\xmapsto{(-)^\dynkin}
		\begin{tikzpicture}[scale=1, anchorbase]
			\draw[thick,line cap=round] (0, -1) to[out=90, in = 180] (0.25, -0.75) to[out=0, in = 90] (0.5,-1);
			\draw[thick,line cap=round] (-0.5, -1) to[out=90, in = 270] (0,0);
		\end{tikzpicture}
		\xmapsto{(-)^*}
		\begin{tikzpicture}[scale=1, anchorbase]
			\draw[thick,line cap=round] (-0.5, 1) to[out=270, in = 180] (-0.25, 0.75) to[out=0, in = 270] (0,1);
			\draw[thick,line cap=round] (0.5, 1) to[out=270, in = 90] (0,0);
		\end{tikzpicture}
		\xmapsto{\cable_{K^!}}
		\begin{tikzpicture}[scale=0.5, anchorbase]
			\fill[green!10] (0,-1) rectangle (1,1);
			\draw[->] (-1,1) -- (-1,-1);
			\draw[->] (0,-1) -- (0,1);
			\draw[->] (1,1) -- (1,-1);
			\draw[->, blue] (1,0.5) -- (0.5,0.5);
		\end{tikzpicture}
	\end{gather*}
\end{proof}

\begin{corollary}
	\label{cor:KHmirrorandduality}
	Let $L$ be a framed, oriented link with $r$ ordered components and $C_{\mathbf{n}}$ as above. Over a field, we have an isomorphism:
	\[ \Kh_L(C_\mathbf{n})^* \cong \Kh_{L^!}(C_\mathbf{n}^*) \]
	
\end{corollary}
\begin{proof}
	
	The natural isomorphism from \autoref{lem:mirrorandduality} extends to the additive closure $\Mat( \Z\mathrm{Link}_{/t})$ and chain complexes over it. Then, the isomorphism $C_\mathbf{n} \cong (C_\mathbf{n})^{\leftrightarrow^{\times r}}$ yields $\cable_{L}(C_\mathbf{n})^* \cong \cable_{L^!}(C_\mathbf{n}^*)$.
	Finally, the monoidal functor $\Kh$ preserves duality and we obtain the claim after rewriting $\Kh_L = \Kh \circ \ \cable_{L}$.
\end{proof}

Now, we will start the proof of \autoref{prop:definitionsEquivalent}. Note that the versions in \autoref{def:coloredHomology} come in pairs of dual concepts---invariants and co-invariants, cochain complex and its dual complex, kernel and cokernel, image and coimage. As a first step, we will investigate the relationships within these pairs.

\begin{proposition}
	\label{prop:dualities}
	In the setting of \autoref{def:coloredHomology} when working over a field of characteristic zero, we have the following isomorphisms: (\invlabel) $\cong$ (\coinvlabel), (\kerlabel) $\cong$ (\cohomlabel) $\cong$ (\imlabel) $\cong$ (\homlabel) $\cong$ (\cokerlabel) and (\imlabel) $\cong$ (\coimlabel).
\end{proposition}
\begin{proof}
	We go through the definitions separately.
	\paragraph{(\invlabel) $\cong$ (\coinvlabel):}
		More generally, for any finite group $G$ and any finite-dimensional $G$-representation $M$ over $\F$ the co-invariants $M_G$ and the invariants $M^G$ are isomorphic via the norm map
		\begin{align*}
		N\colon M_G &\to M^G \\
		[m] &\mapsto \frac{1}{|G|} \sum_{g\in G}gm 
		\end{align*}
		which averages over the group. The map in the other direction is given by the canonical projection $M \to M_G$.
		\paragraph{(\kerlabel) $\cong$ (\cohomlabel) $\cong$ (\imlabel) $\cong$ (\homlabel) $\cong$ (\cokerlabel):}		
		Over a field of characteristic zero, we can apply \autoref{lem:hte} and use that the functor $\Kh_L$ preserves homotopy equivalences as noted earlier to see that $\Kh_L(C_\mathbf{n})$ is homotopy equivalent to the one-term complex having $\Kh_L(\karjw[\mathbf{n}])$ in degree zero. Since every object in $\TL$ is self-dual and the Jones--Wenzl projectors $p_n$ are symmetric under both vertical and horizontal reflection, the object $\karjw \in \Kar(\TL)$ as well as the chain complex having $\karjw$ in degree zero are self-dual, too. In particular, $C_\mathbf{n} \simeq \karjw \cong \karjw^* \simeq C_\mathbf{n}^*$. Again, $\Kh_L$ preserves this homotopy equivalence and therefore, both $\Kh_L(C_\mathbf{n})$ and $\Kh_L(C_\mathbf{n}^*)$ are homotopy equivalent to the complex having $\Kh_L(\karjw[\mathbf{n}])$ in homological degree zero. Forgetting about the trivial homological degree yields the desired isomorphisms since cohomology is concentrated in degree zero.

		\paragraph{(\imlabel) $\cong$ (\coimlabel):}
		This is an immediate consequence of the standard isomorphism $V/\ker f \cong \im f$ for a linear map $f\colon V \to W$ between vector spaces when applied to the projection $p_n$. Interestingly, this argument even holds when working over a commutative ring $\commR$, but we need a field of characteristic zero to have all the Jones--Wenzl projectors $p_n$.\qedhere 
\end{proof}

\begin{remark}
	Understanding the consequences of the missing of Jones--Wenzl projectors over a field of characteristic $p$ is an intricate question of its own. Still, the proof above holds true if the $n$-th Jones--Wenzl projector $p_n$ can be defined over a field $\K$ of characteristic $p$. This is the case when $n < p$ or (miraculously) if $n = ap^k - 1$ for some $k \in \N$ and $1 \leq a < p$. Readers interested in the theory of Jones--Wenzl projectors in positive characteristic are referred to \cite{burrull2019pjoneswenzlidempotents, Tubbenhauer2021-kv} and references therein.
\end{remark}

Finally, we need this well-known observation about Jones--Wenzl projectors that yields the last remaining isomorphism for the proof of \autoref{prop:definitionsEquivalent}.
\begin{proposition}
	 In the setting of \autoref{def:coloredHomology} when working over a field of characteristic zero, we have (\invlabel) $\cong$ (\imlabel).
\end{proposition}
\begin{proof}
	We argue link component-wise: As in \autoref{prop:kerandinv}, we use that the crossing in $\TL$ is given by 
	$\begin{tikzpicture}[anchorbase,scale=0.15]
	\draw[usual] (-1,-1) to [out=90,in=270] (1,1);
	\draw[usual]  (1,-1)  to  [out=90,in=270] (-1,1);
\end{tikzpicture}
=
\begin{tikzpicture}[anchorbase, scale=0.15,tinynodes]
	\draw[usual,out = 90, in = 270] (-1,-1) to (-1,1);
	\draw[usual, out = 90, in = 270]  (1,-1)  to (1,1);
\end{tikzpicture}
\ - \
\begin{tikzpicture}[anchorbase, scale=0.15,tinynodes]
	\draw[usual] (-1,-1) to[out = 90, in = 180] (0,-0.25) to [out = 0, in = 90]  (1,-1);
	\draw[usual]  (-1,1)  to[out = 270, in = 180] (0,0.25) to [out = 0, in = 270] (1,1);
\end{tikzpicture}
$ 
and obtain (in the notation of \autoref{prop:kerandinv}) from the second property of the Jones--Wenzl projector as introduced in \autoref{def:JW} that
\begin{gather*}
		\begin{tikzpicture}[scale=0.15, anchorbase]
		\draw[usual] (-1.5,-2.5) to (-1.5,-1) node[left,xshift=2pt]{$i$} to (-1.5,0.5);
		\draw[usual] (-1,0.5) to[out = 270, in = 180] (0,-0.5) to[out = 0, in = 270] (1,0.5);
		\draw[usual] (-1,-2.5) to[out = 90, in = 180] (0,-1.5) to[out = 0, in =90] (1,-2.5);
		\draw[usual] (1.5,-2.5) to (1.5,-1) node[right,xshift=-2pt]{} to (1.5,0.5);
		\draw[JW,fill=white] (-2.5,-2.5) rectangle (2.5,-4.5) node[pos=0.5] {$n$};
	\end{tikzpicture}
	= 0
	\iff
\begin{tikzpicture}[scale=0.15, anchorbase]
	\draw[usual] (-1.5,-2.5) to (-1.5,-1) node[left,xshift=2pt]{$i$} to (-1.5,0.5);
	\draw[usual] (-1,0.5) to[out = 270, in = 90]  (1,-2.5);
	\draw[usual] (-1,-2.5) to[out = 90, in = 270]  (1,0.5);
	\draw[usual] (1.5,-2.5) to (1.5,-1) node[right,xshift=-2pt]{} to (1.5,0.5);
	\draw[JW,fill=white] (-2.5,-2.5) rectangle (2.5,-4.5) node[pos=0.5] {$n$};
\end{tikzpicture}
=
\begin{tikzpicture}[scale=0.15, anchorbase]
	\draw[JW,fill=white] (-2.5,-2.5) rectangle (2.5,-4.5) node[pos=0.5] {$n$};
\end{tikzpicture}
\end{gather*}
Therefore, $p_n$ is the symmetrizing idempotent for the $S_n$-action. Extending this argument component-wise covers the case of $p_\mathbf{n}$ and the $S_\mathbf{n}$-action. This argument even shows that (\invlabel) = (\imlabel) since they are the same subspace of $\Kh\left(L^{\mathbf{n}}\right)$.
\end{proof}
\begin{proof}[Proof of \autoref{prop:definitionsEquivalent}]
	
	Combining the propositions in this section, we have found all the isomorphisms we need. To compute the Poincaré polynomial, we first note that $\Kh_L(C_\mathbf{n})$ is a bounded trigraded complex of bigraded vector spaces with gradings $q$ and $t$ with homological grading denoted $h$. The differential raises the $h$-grading, but leaves the other two invariant. We can calculate its Euler characteristic $\chi$ as an alternating sum over $h$ and obtain a Laurent polynomial in the variables $q, t$. This yields
	\begin{align*}
		\chi\left( \Kh_L(C_\mathbf{n})\right) (t,q)  &= \sum_{h = 0}^\abs{\lfloor \frac{\mathbf{n}}{2} \rfloor} \sum_{i,j} (-1)^{h} \dim_\F\left( \left(\Kh_L(C_\mathbf{n})^{h}\right)^{i,j} \right) t^i q^j \\
		 &= \sum_{\mathbf{k} = \mathbf{0}}^{\lfloor \frac{\mathbf{n}}{2} \rfloor} \sum_{i,j} (-1)^\abs{\mathbf{k}} \dim_\F\left(\left( \Kh\left(L^{\mathbf{n}-2\mathbf{k}} \right)^{\oplus \binom{\mathbf{n}-\mathbf{k}}{\mathbf{k}}} \right)^{i,j}\right) t^i q^j \\
		 &= \sum_{\mathbf{k} = \mathbf{0}}^{\lfloor \frac{\mathbf{n}}{2} \rfloor} (-1)^\abs{\mathbf{k}} \binom{\mathbf{n}-\mathbf{k}}{\mathbf{k}}  \Pol \left( \Kh \left( L^{\mathbf{n}-2\mathbf{k}} \right) \right) (t,q)
	\end{align*}
	where the superscript $i,j$ denotes the subspace in bigrading $i,j$. On the other hand, we know that $\Kh_L(C_\mathbf{n})$ is homotopy equivalent to a chain complex having $\Kh_L(\karjw[\mathbf{n}])$ in degree zero. The Euler characteristic of the latter complex is precisely the Poincaré polynomial of $\Kh_L(\karjw[\mathbf{n}])$, but since  Euler characteristics agree for homotopy equivalent complexes, we can deduce that
	\[ \Pol\left(\cKh{\mathbf{n}}(L)\right) = \Pol\left(\Kh_L(\karjw[\mathbf{n}])\right) = \chi\left( \Kh_L(C_\mathbf{n})\right) \]
	as claimed.
\end{proof}

\section{Observations}
\label{sec:observations}

Finally, we want to present two observations that can be extracted from the database \cite{coloredKh}: An apparent three-term recursion for $n$-colored Khovanov homology of the one-framed unknot yielding a conjectural closed formula for the Poincaré series of the skein lasagna module of $\overline{\mathbb{CP}^2}$ and the behavior of colored Khovanov homology under twist insertions which leads to the definition of a new family of knot-invariants.
\subsection{One-framed unknot}

The first observation in the data concerns the one-framed unknot. In order to state the observation, it is convenient to define the functions $o, \ell: \N \to \N$ by
\[
o(n)=  \begin{cases}
	2k(k+2) & \text{ if } n = 2k \text{ even } \\
	2k(k+3)+1 & \text{ if } n = 2k+1 \text{ odd }
\end{cases}
\quad \text{and} \quad
\ell(n) = o(n) - o(n-1) =  \begin{cases}
	n+3 & \text{ if } n \text{ even } \\
	n & \text{ if } n \text{ odd }
\end{cases} 
\]
\begin{observation}
	\label{obs:1funknot}
	Let $U_1$ be the 1-framed unknot and define $P_n = \Pol(\cKh{n}(U_1))$ as the Poincaré polynomial of $n$-colored Khovanov homology over $\mathbb{Q}$. For $n \leq 8$, the $n$-colored Khovanov homology is concentrated between bidegrees $(-2\lfloor\frac{n^2}{4}\rfloor, -o(n))$ and $(0, -n)$ if $n$ is even (resp. $(0,-n+2)$ if $n$ is odd) and there is a three-term recursion
	\[ P_n = t^{-2 \lfloor\frac{n}{2}\rfloor } q^{-\ell(n)}P_{n-1} + q^{-2}P_{n-2} + t^{-2 \lfloor\frac{n}{2}\rfloor +1} q^{-\ell(n)+2} P_{n-3} \]
	of Poincaré polynomials of colored Khovanov homology. In particular, the total dimension of $\cKh{n}(U_1)$ is computed by (shifted) Tribonacci numbers \cite[A001590]{oeis}.
	
	Further, the terms corresponding to $P_{n-1}$ and $P_{n-3}$ assemble into the Poincaré polynomial of the (shifted) Khovanov homology of the torus knot $T(n,n-1)$, i.e.\
	\[ t^{-2\lfloor\frac{n^2}{4}\rfloor} q^{-o(n)-(n-1)(n-2)+1} \Pol(\Kh(T(n,n-1))) =  t^{-2 \lfloor\frac{n}{2}\rfloor } q^{-\ell(n)}P_{n-1} + t^{-2 \lfloor\frac{n}{2}\rfloor +1} q^{-\ell(n)+2} P_{n-3}.  \]
\end{observation}
\vspace{-1mm}
	This is illustrated 
	for $n=7$ in \autoref{fig:color7}, highlighting the contributions from $n-3, n-2$ and $n-1$ in red, green and blue respectively. The (shifted) Khovanov homology of $T(n,n-1)$ can be seen as the sum of red and blue parts inside $\cKh{n}(U_1)$.

\begin{figure}[H]
	\centering
	\begin{subfigure}[t]{0.3\textwidth} 
		\centering
		\includegraphics[width=\textwidth]{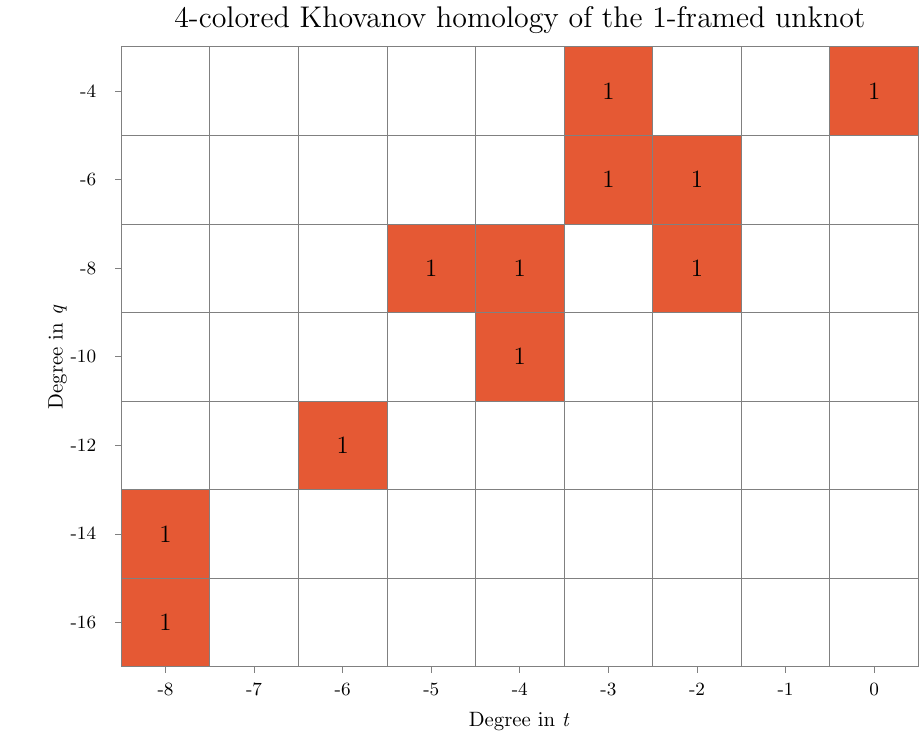}
	\end{subfigure}
	\begin{subfigure}[t]{0.3\textwidth} 
		\centering
		\includegraphics[width=\textwidth]{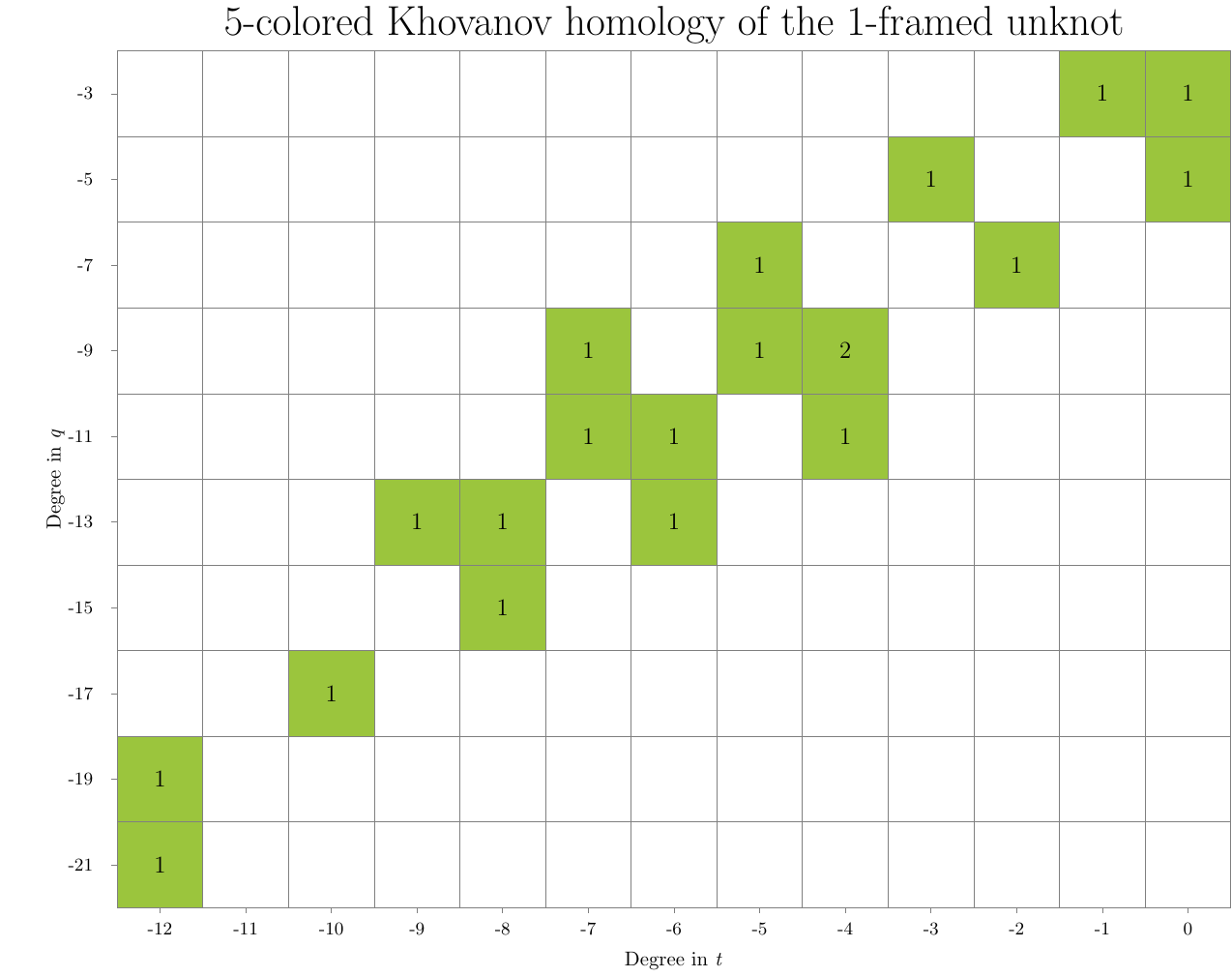}
	\end{subfigure}
	
	\begin{subfigure}[t]{0.465\textwidth} 
		\centering
		\includegraphics[width=\textwidth]{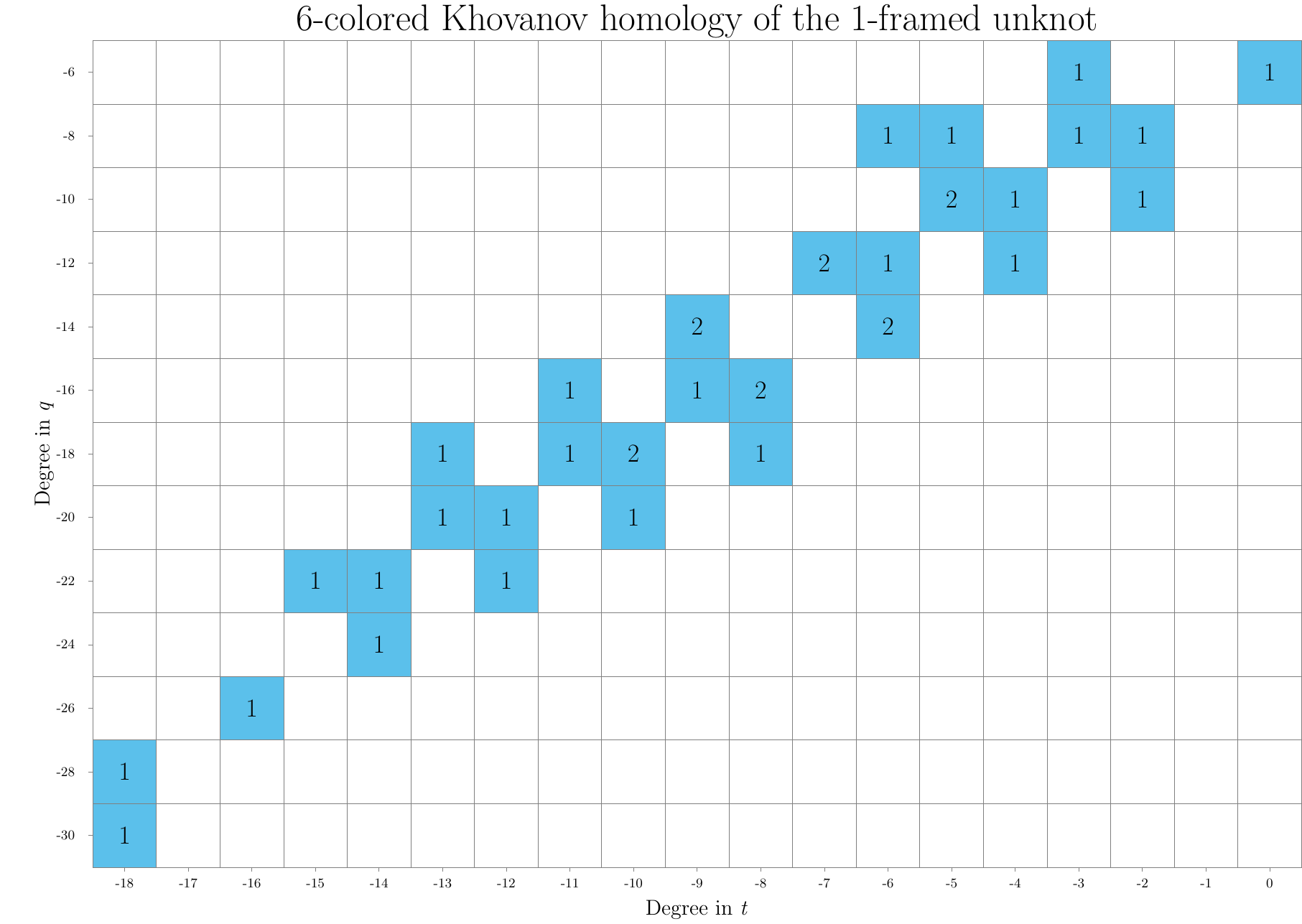}
	\end{subfigure}
	\begin{subfigure}[t]{0.465\textwidth} 
		\centering
		\includegraphics[width=\textwidth]{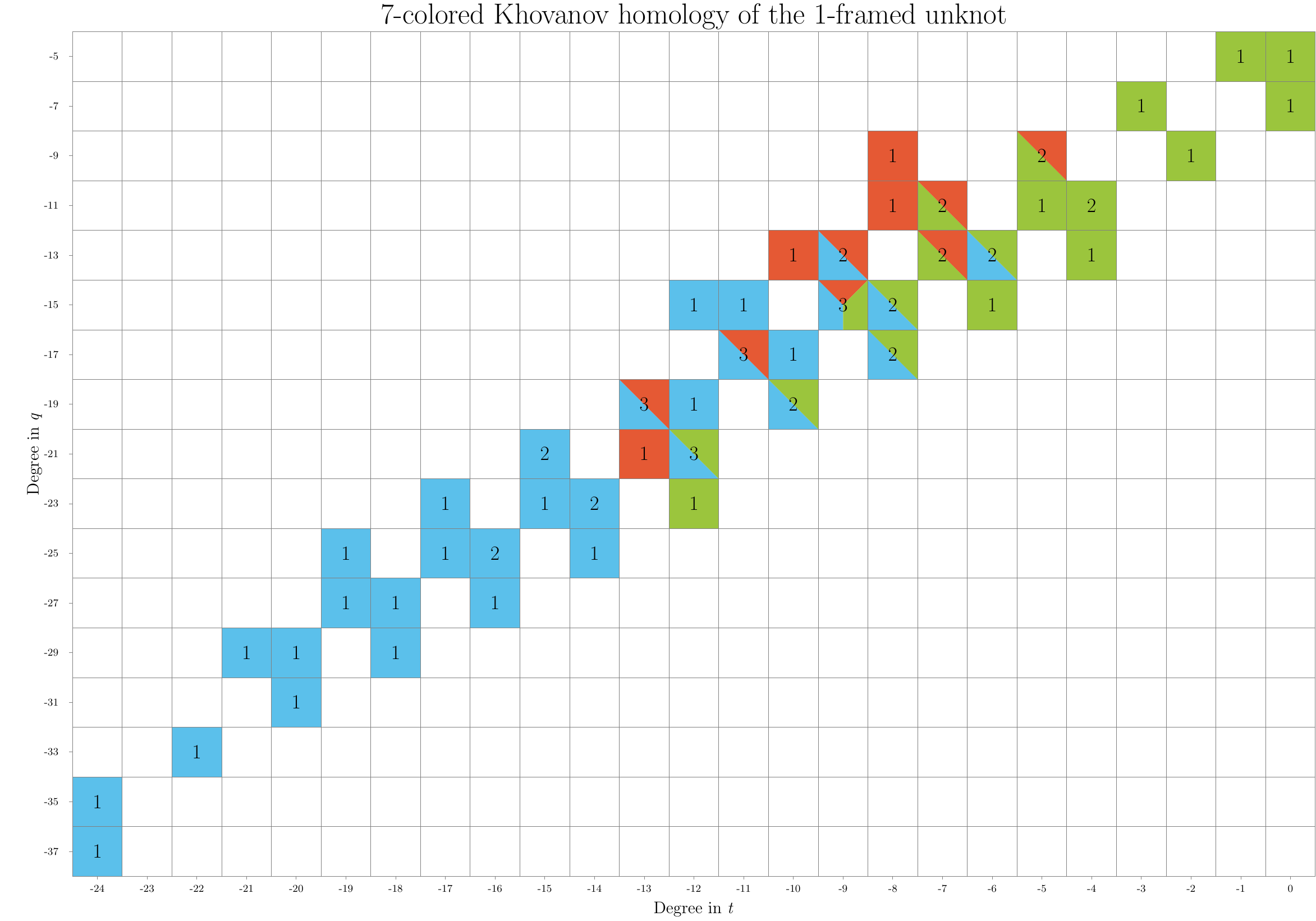}
	\end{subfigure}
	\vspace{-2mm}
	\caption{Poincaré polynomials of (rational) colored Khovanov homology of the 1-framed unknot for $n = 4,5,6,7$. The background color of non-trivial bidegree indicates the contribution to the decomposition of 7-colored homology according to the three-term recursion. The sum of blue and red part corresponds to $t^{-24}q^{-66}\Kh(T(7,6))$.}
	\label{fig:color7}
\end{figure}
\begin{conjecture}
	\label{conj:recursion}
	The formulas from \autoref{obs:1funknot} hold for all $n\in \N_0$.
\end{conjecture}

\autoref{conj:recursion} is related to conjectures of Ren \cite[Conjectures 6.1/6.1']{ren2024leefiltrationstructuretorus} that assert that saddle cobordisms between certain families of generalizations of torus links induce surjective maps on Khovanov homology and by Shumakovitch and Turner \cite[Conjecture 4.4]{gorsky2013stable}/~\cite[Conjecture 6.2]{ren2024leefiltrationstructuretorus} that express the rational Khovanov homologies of $T(n,n-1)$-torus knots in terms of a three-term relation similar to \autoref{conj:recursion}.

\cite[Conjectures 6.1/6.1']{ren2024leefiltrationstructuretorus} imply that the dotted annulus maps expressing the (rational) skein lasagna module of $\overline{\mathbb{CP}^2}$ as a colimit turn out to be all injective (see \cite[Section 6.9]{ren2024khovanovhomologyexotic4manifolds}) and the quotient of $\cKh{n}(K)$ by the image of that map is isomorphic to the rational Khovanov homology of $T(n,n-1)$. Therefore, \cite[Conjecture 4.4]{gorsky2013stable} and \autoref{conj:recursion} are related via this decomposition if \cite[Conjectures 6.1/6.1']{ren2024leefiltrationstructuretorus} hold.
If \autoref{conj:recursion} holds, a calculation similar to \cite[Theorem 4.7]{gorsky2013stable} can be used to find the Poincaré series of the colimits along the inclusions $\cKh{n}(U_1) \hookrightarrow q^2\cKh{n+2}(U_1)$ which correspond (up to adjusting the grading) to the skein lasagna module of $\overline{\mathbb{CP}^2}$ (compare with \cite[Conjecture 6.14]{ren2024khovanovhomologyexotic4manifolds} for an expression in terms of Khovanov homologies of torus knots $T(n,n-1)$).

\begin{conjecture}
	The Poincaré series of the skein lasagna module of $\overline{\mathbb{CP}^2}$ in $H_2$ grading $0$ and $1$ are given by
	\begin{align*}
		\Pol\left(\mathcal{S}_0^2(\overline{\mathbb{CP}^2}; 0; \Q) \right) &= \lim_{n \to \infty} q^{-2n}P_{2n}(t,q^{-1}) \\
		= \frac{1}{1+q^{-4}t}  \sum_{k = 0}^\infty t^{-2k^2} q^{2k(k+1)}   \frac{(-q^{-4}t; q^{2}t^{-2})_{2k+1}}{(q^{2}t^{-2}; q^{2}t^{-2})_{2k}} 
		&+   t^{-2(k+1)^2} q^{2((k+1)^2 + k)}  \frac{(-q^{-4}t; q^{2}t^{-2})_{2k+2}}{(q^{2}t^{-2}; q^{2}t^{-2})_{2k+1}}  \\
		&\text{and} \\
		\Pol\left(\mathcal{S}_0^2(\overline{\mathbb{CP}^2}; 1; \Q) \right) &= \lim_{n \to \infty} q^{-2n-1}P_{2n+1}(t,q^{-1}) \\
		=  \frac{1}{1+q^{-4}t} \sum_{k = 0}^\infty t^{-2k(k+1)} q^{2(k^2+2k-1)}  \frac{(-q^{-4}t; q^{2}t^{-2})_{2k+1}}{(q^{2}t^{-2}; q^{2}t^{-2})_{2k}}  
		&+   t^{-2k(k+1)} q^{2k(k+2)} \frac{(-q^{-4}t; q^{2}t^{-2})_{2k+2}}{(q^{2}t^{-2}; q^{2}t^{-2})_{2k+1}}  \\
	\end{align*}
	\vspace{-11mm}
\end{conjecture}

Additionally, since the relationship between $\cKh{n}(K)$ and $\Kh(K^i)$ (for $0 \leq i \leq n$) is encoded by the Chebyshev polynomials of the second kind, one can invert this relation to express $\Kh(K^n)$ in terms of the $\cKh{i}(K)$ (for $0 \leq i \leq n$). Therefore, if \autoref{conj:recursion} has an affirmative answer, this can be used to compute $\Kh(T(n,n))$ (see also \cite[Conjecture 6.1']{ren2024leefiltrationstructuretorus}).

\subsection{Minimal total dimensions under varying the framing}
\label{sec:examples}

Under a framing change, the colored Jones polynomials of a link change by a monomial factor depending on the color of the relevant component. In contrast, colored Khovanov homology has a more interesting behavior, as already illustrated in the previous section. The data in \cite{coloredKh} for $2$-colored and (small examples of) $3$-colored Khovanov homology suggests that the total dimensions of these homologies show some systematic behavior in the limit case. This motivates the following definition of a whole family of knot invariants for any \emph{unframed}, oriented knot $K$ and questions regarding their properties:

\begin{definition}
	For each knot $K$, let $X(K)\subset \Z \times \N$ be the subset of pairs $(k,n)$, such
	that the dimension of $n$-colored Khovanov homology of the knot is minimal for
	the framing $k$. We can partition the set $X(K)$ by colors and write $X_n(K):=\{k\in \Z\mid (k,n)\in X(K)\}$. Both $X(K)$ and $X_n(K)$ are invariants of the knot $K$.
\end{definition}

\begin{question}
\label{question:limit}
Fix a knot $K$ in $B^3$ and a number $n\in \N_{\geq 2}$. Does the total dimension of
$n$-colored Khovanov homology of the $k$-framed knot $K$ tend to $\infty$ for
$k\to \pm \infty$? 
\end{question}

Using mirror arguments, this is related to the $n$-projector colored homology of $K$ and $K^!$ being infinite  \cite{cooper2012categorification, Rozansky2014-ch}. An affirmative answerer to \autoref{question:limit} would imply that $X_n(K)$ is finite. In this case, we can ask more questions about its structure:
\begin{question}
	\label{question:structure}
	\begin{enumerate}[(a)]
		\item \label{question:interval} Is $X_n(K)$ an interval $[a_n, b_n]$? 
		\item  \label{question:monotone_limit} Is the dimension of $n$-colored Khovanov homology of the knot $K$ with framing $k$ monotonically increasing (resp.\ decreasing) in $k$ for $k \geq k_{+} > b_n$ (resp. $k \leq k_- < a_n$) for some $k_+$ (resp.\ $k_-$)?
		\item \label{question:monotone}  Do the above monotony statements already hold for $k \geq b_n$ (resp.\ for $k \leq a_n$)?
		\item Do the sequences of endpoints $a_n$ and $b_n$, centers $c_n = \frac{a_n+b_n}{2}$ or lengths $\ell_n = b_n - a_n$ have limits as $n\to \infty$?
	\end{enumerate}
\end{question}

\begin{remark}
	If the knot $K$ is amphichiral, it is (possibly up to orientation) equivalent to its mirror. Therefore, $X_n(K)$ is symmetric around $0$ for all $n$. 
\end{remark}

 In the $n=2$ case finiteness of $X_2(K)$ and item (\ref*{question:monotone_limit}) of \autoref{question:structure}  with linear growth of slope $2$ can be easily deduced from the structure of the categorified $2$-projector \cite{cooper2012categorification}. The $2$-projector colored homology of a knot contains (at large enough bidegrees) a subcomplex homotopy equivalent to an infinite sum of shifted copies of
 \[ \Q[X]/(X^2) \xrightarrow{\cdot X} \Q[X]/(X^2) \]
 with two-dimensional homology.
 Additionally, increasing the framing of $K$ by one corresponds  \footnote{up to a change of framing of $K^2$} to inserting an additional full twist to $K^2$. This procedure approximates the $2$-projector colored homology in the large-twist limit \cite{Rozansky2014-ch} by (in the limit) adding one summand of the above form per inserted twist.
 The available data in \cite{coloredKh} suggests an affirmative answer for item (\ref*{question:interval}) with  all computed intervals having length $2$ or $3$ and an extension of the linear growth of slope $2$ to item (\ref*{question:monotone}) as well.

In view of the available data, we ask:

\begin{question}
	Is $X_2(K)$ or more generally $X_n(K)$ a concordance invariant? 
\end{question}
\autoref{tab:data} contains a selection of concordance invariants (the signature, the $4$-ball genus, the concordance order, the $s$ invariant~\cite{Rasmussen2010}, $\tau$~\cite{Ozsv_th_2003}, $\nu$~\cite{Ozsv_th_2010},  $\epsilon$~\cite{Hom_2013} and $\vartheta_0$ \cite{Lewark_2024}). Whenever two knots from the table agree on all of these concordance invariants, they also agree on $X_2$. Additionally, there seems to be a connection between $\vartheta_0$ and $X_2$ since $\vartheta_0(K) \in X_2(K)$ in all instances, but in the cases where $|X_2(K)| = 2$, it is not consistent if $\vartheta_0$ is the smaller or larger one, but when $|X_2(K)| = 3$, we always observe $X_2(K) =\{-1,0,1\}$ and $\vartheta_0(K) = 0$. This in particular includes all knots in the table that are amphichiral or slice.

\begin{table}[H]
	\begin{tabular}{cccccccccccccccc}
		\begin{tabular}{ccccccccccccccccc}
			\toprule
			Name & Braid Notation \tablefootnote{This table shows the knotinfo convention for chiral knots, but \cite{coloredKh} uses the Knot Atlas convention by default. This results in most but not all of these knots to be mirrored.} & $X_2$ & $\vartheta_0$ & $\sigma$ & $g_4$ & Con. Order & $s$ & $\tau$ & $\nu$ & $\epsilon$ \\
			\midrule
			$3_{1}$ & $\{1,1,1\}$ & $\{2, 3\}$ & 3 & -2 & 1 & $\infty$ & 2 & 1 & $\{1,-1\}$ & 1 \\
			$4_{1}$ & $\{1,-2,1,-2\}$ & $\{-1, 0, 1\}$ & 0 & 0 & 1 & 2 & 0 & 0 & $\{0,0\}$ & 0 \\
			$5_{1}$ & $\{1,1,1,1,1\}$ & $\{5, 6\}$ & 6 & -4 & 2 & $\infty$ & 4 & 2 & $\{2,-2\}$ & 1 \\
			$5_{2}$ & $\{1,1,1,2,-1,2\}$ & $\{2, 3\}$ & 3 & -2 & 1 & $\infty$ & 2 & 1 & $\{1,-1\}$ & 1 \\
			$6_{1}$ & $\{1,1,2,-1,-3,2,-3\}$ & $\{-1, 0, 1\}$ & 0 & 0 & 0 & 1 & 0 & 0 & $\{0,0\}$ & 0 \\
			\midrule
			$6_{2}$ & $\{1,1,1,-2,1,-2\}$ & $\{2, 3\}$ & 3 & -2 & 1 & $\infty$ & 2 & 1 & $\{1,-1\}$ & 1 \\
			$6_{3}$ & $\{1,1,-2,1,-2,-2\}$ & $\{-1, 0, 1\}$ & 0 & 0 & 1 & 2 & 0 & 0 & $\{0,0\}$ & 0 \\
			$7_{1}$ & $\{1,1,1,1,1,1,1\}$ & $\{8, 9\}$ & 9 & -6 & 3 & $\infty$ & 6 & 3 & $\{3,-3\}$ & 1 \\
			$7_{2}$ & $\{1,1,1,2,-1,2,3,-2,3\}$ & $\{2, 3\}$ & 3 & -2 & 1 & $\infty$ & 2 & 1 & $\{1,-1\}$ & 1 \\
			$7_{3}$ & $\{1,1,1,1,1,2,-1,2\}$ & $\{5, 6\}$ & 6 & -4 & 2 & $\infty$ & 4 & 2 & $\{2,-2\}$ & 1 \\
			\midrule
			$7_{4}$ & $\{1,1,2,-1,2,2,3,-2,3\}$ & $\{2, 3\}$ & 3 & -2 & 1 & $\infty$ & 2 & 1 & $\{1,-1\}$ & 1 \\
			$7_{5}$ & $\{1,1,1,1,2,-1,2,2\}$ & $\{5, 6\}$ & 6 & -4 & 2 & $\infty$ & 4 & 2 & $\{2,-2\}$ & 1 \\
			$7_{6}$ & $\{1,1,-2,1,3,-2,3\}$ & $\{2, 3\}$ & 3 & -2 & 1 & $\infty$ & 2 & 1 & $\{1,-1\}$ & 1 \\
			$7_{7}$ & $\{-1,2,-1,2,-3,2,-3\}$ & $\{-1, 0, 1\}$ & 0 & 0 & 1 & $\infty$ & 0 & 0 & $\{0,0\}$ & 0 \\
			$8_{1}$ & $\{1,1,2,-1,2,3,-2,-4,3,-4\}$ & $\{-1, 0, 1\}$ & 0 & 0 & 1 & $\infty$ & 0 & 0 & $\{0,0\}$ & 0 \\
			\midrule
			$8_{2}$ & $\{1,1,1,1,1,-2,1,-2\}$ & $\{5, 6\}$ & 6 & -4 & 2 & $\infty$ & 4 & 2 & $\{2,-2\}$ & 1 \\
			$8_{3}$ & $\{1,1,2,-1,-3,2,-3,-4,3,-4\}$ & $\{-1, 0, 1\}$ & 0 & 0 & 1 & 2 & 0 & 0 & $\{0,0\}$ & 0 \\
			$8_{4}$ & $\{-1,-1,-1,2,-1,2,3,-2,3\}$ & $\{-3, -2\}$ & -3 & 2 & 1 & $\infty$ & -2 & -1 & $\{-1,1\}$ & -1 \\
			$8_{5}$ & $\{1,1,1,-2,1,1,1,-2\}$ & $\{5, 6\}$ & 6 & -4 & 2 & $\infty$ & 4 & 2 & $\{2,-2\}$ & 1 \\
			$8_{6}$ & $\{1,1,1,1,2,-1,-3,2,-3\}$ & $\{2, 3\}$ & 3 & -2 & 1 & $\infty$ & 2 & 1 & $\{1,-1\}$ & 1 \\
			\midrule
			$8_{7}$ & $\{-1,-1,-1,-1,2,-1,2,2\}$ & $\{-3, -2\}$ & -3 & 2 & 1 & $\infty$ & -2 & -1 & $\{-1,1\}$ & -1 \\
			$8_{8}$ & $\{-1,-1,-1,-2,1,3,-2,3,3\}$ & $\{-1, 0, 1\}$ & 0 & 0 & 0 & 1 & 0 & 0 & $\{0,0\}$ & 0 \\
			$8_{9}$ & $\{1,1,1,-2,1,-2,-2,-2\}$ & $\{-1, 0, 1\}$ & 0 & 0 & 0 & 1 & 0 & 0 & $\{0,0\}$ & 0 \\
			$8_{10}$ & $\{-1,-1,-1,2,-1,-1,2,2\}$ & $\{-3, -2\}$ & -3 & 2 & 1 & $\infty$ & -2 & -1 & $\{-1,1\}$ & -1 \\
			$8_{11}$ & $\{1,1,2,-1,2,2,-3,2,-3\}$ & $\{2, 3\}$ & 3 & -2 & 1 & $\infty$ & 2 & 1 & $\{1,-1\}$ & 1 \\
			\midrule
			$8_{12}$ & $\{1,-2,1,3,-2,-4,3,-4\}$ & $\{-1, 0, 1\}$ & 0 & 0 & 1 & 2 & 0 & 0 & $\{0,0\}$ & 0 \\
			$8_{13}$ & $\{1,1,-2,1,-2,-2,-3,2,-3\}$ & $\{-1, 0, 1\}$ & 0 & 0 & 1 & $\infty$ & 0 & 0 & $\{0,0\}$ & 0 \\
			$8_{14}$ & $\{1,1,1,2,-1,2,-3,2,-3\}$ & $\{2, 3\}$ & 3 & -2 & 1 & $\infty$ & 2 & 1 & $\{1,-1\}$ & 1 \\
			$8_{15}$ & $\{1,1,-2,1,3,2,2,2,3\}$ & $\{5, 6\}$ & 6 & -4 & 2 & $\infty$ & 4 & 2 & $\{2,-2\}$ & 1 \\
			$8_{16}$ & $\{-1,-1,2,-1,-1,2,-1,2\}$ & $\{-3, -2\}$ & -3 & 2 & 1 & $\infty$ & -2 & -1 & $\{-1,1\}$ & -1 \\
			\midrule
			$8_{17}$ & $\{1,1,-2,1,-2,1,-2,-2\}$ & $\{-1, 0, 1\}$ & 0 & 0 & 1 & 2 & 0 & 0 & $\{0,0\}$ & 0 \\
			$8_{18}$ & $\{1,-2,1,-2,1,-2,1,-2\}$ & {\color{red}$\{-1, 0, 1\}$} & 0 & 0 & 1 & 2 & 0 & 0 & $\{0,0\}$ & 0 \\
			$8_{19}$ & $\{1,1,1,2,1,1,1,2\}$ & $\{7, 8\}$ & 8 & -6 & 3 & $\infty$ & 6 & 3 & $\{3,-2\}$ & 1 \\
			$8_{20}$ & $\{1,1,1,-2,-1,-1,-1,-2\}$ & $\{-1, 0, 1\}$ & 0 & 0 & 0 & 1 & 0 & 0 & $\{0,0\}$ & 0 \\
			$8_{21}$ & $\{1,1,1,2,-1,-1,2,2\}$ & $\{2, 3\}$ & 3 & -2 & 1 & $\infty$ & 2 & 1 & $\{1,0\}$ & 1 \\
			\bottomrule
		\end{tabular}
	\end{tabular}
	\caption{Table of knot concordance invariants and $X_2$ for all prime knots with up to 8 crossings. Red values are extrapolated from incomplete computational data.}
	\label{tab:data}
\end{table}

\bibliographystyle{alpha}
\bibliography{refs}

\begin{thebibliography}{RvMW24}

\bibitem[BHPW25]{beliakova2025unificationcoloredannularsl2}
Anna Beliakova, Matthew Hogancamp, Krzysztof Putyra, and Stephan Wehrli.
\newblock On unification of colored annular sl2 knot homology.
\newblock {\em Advances in Mathematics}, 468:110206, 2025.

\bibitem[Bla10]{Blanchet2010}
Christian Blanchet.
\newblock An oriented model for {K}hovanov homology.
\newblock {\em J. Knot Theory Ramifications}, 19(2):291--312, 2010.

\bibitem[BLS19]{burrull2019pjoneswenzlidempotents}
Gaston Burrull, Nicolas Libedinsky, and Paolo Sentinelli.
\newblock p-{J}ones-{W}enzl idempotents.
\newblock {\em Advances in Mathematics}, 352:246--264, 2019.

\bibitem[BN02]{Bar_Natan_2002}
Dror Bar-Natan.
\newblock On {K}hovanov’s categorification of the {J}ones polynomial.
\newblock {\em Algebraic \& Geometric Topology}, 2(1):337–370, May 2002.

\bibitem[BN05]{Bar_Natan_2005}
Dror Bar-Natan.
\newblock Khovanov’s homology for tangles and cobordisms.
\newblock {\em Geometry \& Topology}, 9(3):1443–1499, August 2005.

\bibitem[BN07]{barnatan2007fastkhovanovhomologycomputations}
Dror Bar-Natan.
\newblock Fast {K}hovanov {H}omology {C}omputations.
\newblock {\em Journal of Knot Theory and Its Ramifications}, 16(03):243--255,
  2007.

\bibitem[BW08]{Beliakova_Wehrli_2008}
Anna Beliakova and Stephan Wehrli.
\newblock Categorification of the {C}olored {J}ones {P}olynomial and
  {R}asmussen {I}nvariant of {L}inks.
\newblock {\em Canadian Journal of Mathematics}, 60(6):1240–1266, 2008.

\bibitem[Cap08]{Caprau2008}
Carmen~Livia Caprau.
\newblock {$\rm sl(2)$} tangle homology with a parameter and singular
  cobordisms.
\newblock {\em Algebr. Geom. Topol.}, 8(2):729--756, 2008.

\bibitem[CE24]{Coulembier2024Nsphericalfunctorsandtensorcategories}
Kevin Coulembier and Pavel Etingof.
\newblock N-spherical {F}unctors and {T}ensor {C}ategories.
\newblock {\em International Mathematics Research Notices},
  2024(14):10615--10649, 05 2024.

\bibitem[CK12]{cooper2012categorification}
Benjamin Cooper and Vyacheslav Krushkal.
\newblock Categorification of the {J}ones--{W}enzl projectors.
\newblock {\em Quantum Topology}, 3(2):139--180, 2012.

\bibitem[CMW09]{Clark_2009}
David Clark, Scott Morrison, and Kevin Walker.
\newblock Fixing the functoriality of {K}hovanov homology.
\newblock {\em Geometry \& Topology}, 13(3):1499–1582, March 2009.

\bibitem[DKS23]{dyckerhoff2023nsphericalfunctorscategorificationeulers}
Tobias Dyckerhoff, Mikhail Kapranov, and Vadim Schechtman.
\newblock N-spherical functors and categorification of {E}uler's continuants,
  2023.
\newblock \arxiv{2306.13350}.

\bibitem[DW25]{dyckerhoff2025perverseschoberscoxetertype}
Tobias Dyckerhoff and Paul Wedrich.
\newblock Perverse schobers of {C}oxeter type $\mathbb{A}$, 2025.
\newblock \arxiv{2504.08496}.

\bibitem[ETW18]{Ehrig_2018}
Michael Ehrig, Daniel Tubbenhauer, and Paul Wedrich.
\newblock Functoriality of colored link homologies.
\newblock {\em Proceedings of the London Mathematical Society},
  117(5):996–1040, June 2018.

\bibitem[FSS12]{Frenkel2012-cb}
Igor~B Frenkel, Catharina Stroppel, and Joshua Sussan.
\newblock Categorifying fractional {E}uler characteristics, {Jones--Wenzl}
  projectors and $3j$-symbols.
\newblock {\em Quantum Topol.}, 3(2):181--253, March 2012.

\bibitem[Gar04]{garoufalidis2004characteristicdeformationvarietiesknot}
Stavros Garoufalidis.
\newblock On the characteristic and deformation varieties of a knot.
\newblock {\em Geom. Topol. Monogr}, 7:291--304, 2004.

\bibitem[Gar11]{Garoufalidis2011-vq}
Stavros Garoufalidis.
\newblock The {J}ones slopes of a knot.
\newblock {\em Quantum Topol.}, 2(1):43--69, January 2011.

\bibitem[GLW17]{Grigsby_2017}
J.~Elisenda Grigsby, Anthony~M. Licata, and Stephan~M. Wehrli.
\newblock Annular {K}hovanov homology and knotted {S}chur–{W}eyl
  representations.
\newblock {\em Compositio Mathematica}, 154(3):459–502, November 2017.

\bibitem[GOR13]{gorsky2013stable}
Eugene Gorsky, Alexei Oblomkov, and Jacob Rasmussen.
\newblock On {S}table {K}hovanov {H}omology of {T}orus {K}nots.
\newblock {\em Experimental Mathematics}, 22(3):265--281, 2013.

\bibitem[GPPV20]{gukov2020bpsspectra3manifoldinvariants}
Sergei Gukov, Du~Pei, Pavel Putrov, and Cumrun Vafa.
\newblock {BPS} spectra and 3-manifold invariants.
\newblock {\em Journal of Knot Theory and Its Ramifications}, 29(02):2040003,
  2020.

\bibitem[GW23]{Gorsky2023-ih}
Eugene Gorsky and Paul Wedrich.
\newblock Evaluations of annular {Khovanov--Rozansky} homology.
\newblock {\em Math. Z.}, 303(1), January 2023.

\bibitem[Hog18]{Hogancamp2018-xq}
Matthew Hogancamp.
\newblock A polynomial action on colored $\mathfrak {sl}_2$ link homology.
\newblock {\em Quantum Topol.}, 10(1):1--75, October 2018.

\bibitem[Hog20]{hogancamp2020homological}
Matthew Hogancamp.
\newblock Homological perturbation theory with curvature, 2020.
\newblock \arxiv{1912.03843}.

\bibitem[Hom13]{Hom_2013}
Jennifer Hom.
\newblock Bordered {H}eegaard {F}loer homology and the tau-invariant of cable
  knots.
\newblock {\em Journal of Topology}, 7(2):287–326, August 2013.

\bibitem[HRW25]{Hogancamp2025-ol}
Matthew Hogancamp, David E~V Rose, and Paul Wedrich.
\newblock {A} {K}irby color for {K}hovanov homology.
\newblock {\em J. Eur. Math. Soc.}, February 2025.

\bibitem[Hsu93]{205649}
W.-J. Hsu.
\newblock Fibonacci cubes-a new interconnection {T}opology.
\newblock {\em IEEE Transactions on Parallel and Distributed Systems},
  4(1):3--12, 1993.

\bibitem[Kas97]{kashaev1997hyperbolic}
Rinat~M Kashaev.
\newblock The hyperbolic volume of knots from the quantum dilogarithm.
\newblock {\em Letters in mathematical physics}, 39(3):269--275, 1997.

\bibitem[Kho00]{khovanov2000categorification}
Mikhail Khovanov.
\newblock A categorification of the {J}ones polynomial.
\newblock {\em Duke Math. J.}, 104(1):359--426, 2000.

\bibitem[Kho02]{Khovanov_2002}
Mikhail Khovanov.
\newblock A functor-valued invariant of tangles.
\newblock {\em Algebraic \& Geometric Topology}, 2(2):665–741, September
  2002.

\bibitem[Kho03]{khovanov2003patterns}
Mikhail Khovanov.
\newblock Patterns in knot cohomology, {I}.
\newblock {\em Experimental mathematics}, 12(3):365--374, 2003.

\bibitem[Kho05]{khovanov2005categorifications}
Mikhail Khovanov.
\newblock Categorifications of the colored {J}ones polynomial.
\newblock {\em Journal of Knot Theory and Its Ramifications}, 14(01):111--130,
  2005.

\bibitem[KL94]{Kauffman-Lins94}
Louis~H. Kauffman and Sóstenes~L. Lins.
\newblock {\em Temperley-{L}ieb {R}ecoupling {T}heory and {I}nvariants of
  3-Manifolds (AM-134)}.
\newblock Princeton University Press, 1994.

\bibitem[KM11]{kronheimer2011khovanov}
P.~B. Kronheimer and T.~S. Mrowka.
\newblock Khovanov homology is an unknot-detector.
\newblock {\em Publications Math\'ematiques de l'IH\'ES}, 113:97--208, 2011.

\bibitem[Knoa]{Knotatlas}
{The Knot Atlas}.
\newblock \url{https://katlas.org/}.
\newblock [Online; accessed 23-April-2025].

\bibitem[Knob]{KnotTheory}
{The Mathematica Package KnotTheory`}.
\newblock \url{https://katlas.org/wiki/The_Mathematica_Package_KnotTheory%60}.
\newblock [Online; accessed 23-April-2025].

\bibitem[KR08]{khovanov2008matrix}
Mikhail Khovanov and Lev Rozansky.
\newblock Matrix factorizations and link homology.
\newblock {\em Fundamenta mathematicae}, 199(1):1--91, 2008.

\bibitem[KRSS17]{Kucharski_2017}
Piotr Kucharski, Markus Reineke, Marko Stošić, and Piotr Sułkowski.
\newblock {BPS} states, knots, and quivers.
\newblock {\em Physical Review D}, 96(12), December 2017.

\bibitem[KRSS19]{Kucharski_2019}
Piotr Kucharski, Markus Reineke, Marko Stošić, and Piotr Sułkowski.
\newblock Knots-quivers correspondence.
\newblock {\em Advances in Theoretical and Mathematical Physics},
  23(7):1849–1902, 2019.

\bibitem[Lew]{Khoca}
Lukas Lewark.
\newblock {Khoca, a knot homology calculator}.
\newblock \url{https://people.math.ethz.ch/~llewark/khoca.php}.
\newblock [Online; accessed 23-April-2025].

\bibitem[LL16]{Lewark_2016}
Lukas Lewark and Andrew Lobb.
\newblock New quantum obstructions to sliceness.
\newblock {\em Proceedings of the London Mathematical Society},
  112(1):81–114, January 2016.

\bibitem[LZ24]{Lewark_2024}
Lukas Lewark and Claudius Zibrowius.
\newblock Rasmussen invariants of {W}hitehead doubles and other satellites.
\newblock {\em Journal für die reine und angewandte Mathematik (Crelles
  Journal)}, September 2024.

\bibitem[Mar01]{markl2001ideal}
Martin Markl.
\newblock Ideal {P}erturbation {L}emma.
\newblock {\em Communications in Algebra}, 29(11):5209--5232, 2001.

\bibitem[MM01]{murakami2001colored}
Hitoshi Murakami and Jun Murakami.
\newblock The colored {J}ones polynomials and the simplicial volume of a knot.
\newblock {\em Acta Mathematica}, 186(1):85--104, 2001.

\bibitem[MN22]{manolescu2022skein}
Ciprian Manolescu and Ikshu Neithalath.
\newblock Skein lasagna modules for 2-handlebodies.
\newblock {\em Journal für die reine und angewandte Mathematik (Crelles
  Journal)}, 2022(788):37--76, 2022.

\bibitem[MWW22]{Morrison_2022}
Scott Morrison, Kevin Walker, and Paul Wedrich.
\newblock Invariants of 4–manifolds from {Khovanov–Rozansky} link homology.
\newblock {\em Geometry \& Topology}, 26(8):3367–3420, December 2022.

\bibitem[MWW23]{MR4589588}
Ciprian Manolescu, Kevin Walker, and Paul Wedrich.
\newblock Skein lasagna modules and handle decompositions.
\newblock {\em Adv. Math.}, 425:Paper No. 109071, 40, 2023.

\bibitem[{OEI}25]{oeis}
{OEIS Foundation Inc.}
\newblock The {O}n-{L}ine {E}ncyclopedia of {I}nteger {S}equences, 2025.
\newblock Published electronically at \url{http://oeis.org}.

\bibitem[OS03]{Ozsv_th_2003}
Peter Ozsváth and Zoltán Szabó.
\newblock Knot {F}loer homology and the four-ball genus.
\newblock {\em Geometry \& Topology}, 7(2):615–639, October 2003.

\bibitem[OS10]{Ozsv_th_2010}
Peter S Ozsváth and Zoltán Szabó.
\newblock Knot {F}loer homology and rational surgeries.
\newblock {\em Algebraic \& Geometric Topology}, 11(1):1–68, January 2010.

\bibitem[Ras05]{rasmussen2005khovanovsinvariantclosedsurfaces}
Jacob Rasmussen.
\newblock Khovanov's invariant for closed surfaces, 2005.
\newblock \arxiv{math/0502527}.

\bibitem[Ras10]{Rasmussen2010}
Jacob Rasmussen.
\newblock Khovanov homology and the slice genus.
\newblock {\em Inventiones mathematicae}, 182(2):419–447, September 2010.

\bibitem[Ren24]{ren2024leefiltrationstructuretorus}
Qiuyu Ren.
\newblock Lee filtration structure of torus links.
\newblock {\em Geom. Topol.}, 28(8):3935--3960, 2024.

\bibitem[Roz14]{Rozansky2014-ch}
Lev Rozansky.
\newblock An infinite torus braid yields a categorified {Jones--Wenzl}
  projector.
\newblock {\em Fund. Math.}, 225(1):305--326, 2014.

\bibitem[RT91]{Reshetikhin1991}
N.~Reshetikhin and V.~G. Turaev.
\newblock Invariants of 3-manifolds via link polynomials and quantum groups.
\newblock {\em Inventiones mathematicae}, 103(1):547--597, Dec 1991.

\bibitem[RvMW24]{coloredKh}
Karim Ritter~von Merkl and Paul Wedrich.
\newblock Colored-kh, a database for colored {K}hovanov homology.
\newblock \url{https://colored-kh.math.uni-hamburg.de}, 2024.

\bibitem[RW24]{ren2024khovanovhomologyexotic4manifolds}
Qiuyu Ren and Michael Willis.
\newblock Khovanov homology and exotic $4$-manifolds, 2024.
\newblock \arxiv{2402.10452}.

\bibitem[Sch]{Knotjob}
Dirk Schütz.
\newblock {KnotJob}.
\newblock \url{https://www.maths.dur.ac.uk/users/dirk.schuetz/knotjob.html}.
\newblock [Online; accessed 23-April-2025].

\bibitem[SZ24]{sullivan2024kirbybeltscategorifiedprojectors}
Ian Sullivan and Melissa Zhang.
\newblock Kirby belts, categorified projectors, and the skein lasagna module of
  {$S^{2} \times S^{2}$}.
\newblock {\em Quantum Topol.}, December 2024.

\bibitem[TW21]{Tubbenhauer2021-kv}
Daniel Tubbenhauer and Paul Wedrich.
\newblock Quivers for {SL}(2) tilting modules.
\newblock {\em Represent. Theory}, 25(15):440--480, June 2021.

\end{thebibliography}

\end{document}